\documentclass[a4paper]{amsart}

\usepackage{config}
\usepackage{commands}

\title{Moduli spaces of nilpotent displays}

\author{Sebastian Bartling}
\address{Sebastian Bartling: Fakultät für Mathematik, Universität Duisburg-Essen, D-45127 Essen}
\email{sebastian-bartling@hotmail.de}

\author{Manuel Hoff}
\address{Manuel Hoff: Fakultät für Mathematik, Universität Bielefeld, D-33501 Bielefeld}
\email{manuel.hoff@uni-bielefeld.de}

\begin{document}

    \begin{abstract}
        We show that the moduli problem of deformations of nilpotent displays by quasi-isogenies is representable, without using $p$-divisible groups.
        The main ingredients are Artin's criterion and the theory of truncated displays.
        This gives in particular a new proof for the representability of Rapoport-Zink spaces.
    \end{abstract}

	\maketitle

    {
    \hypersetup{hidelinks}
    \tableofcontents
    }
    
    \section{Introduction}
    Rapoport-Zink spaces \cite{RZ} are moduli spaces of $p$-divisible groups (with additional structure).
The idea that generic fibers of Rapoport-Zink spaces should fit into a general group-theoretic framework of local Shimura varieties was put forward by Rapoport-Viehmann \cite{RapoViehmann} and into reality by Scholze-Weinstein \cite{Berkeleylectures}.
The original moduli problem of $p$-divisible groups is then considered as an integral model of a local Shimura variety and recently there has been work on generalizing Rapoport-Zink spaces to a group-theoretic setting (\cite{KimHodge} \cite{HowardPappas} \cite{BueltelPappas} \cite{XuShenRZ} \cite{HamacherKim} \cite{RapoPappasIntegralModels}).
As far as we are aware of, all constructions of integral models of local Shimura varieties found in the literature rely via intricate steps on the original representability result of Rapoport-Zink and therefore on the use of $p$-divisible groups.

The aim of this article is to give a new proof of the representability of Rapoport-Zink spaces for the group $\GL_h$ using Zink's theory of displays which is completely independent of $p$-divisible groups.

\medskip

Let us briefly recall the classical Rapoport-Zink moduli problem of $p$-divisible groups considered in \cite{RZ}.
Fix a prime $p$ and let $\F$ be an algebraically closed field of characteristic $p$.
Write $W(\F)$ for the ring of Witt vectors of $\F$, fix a $p$-divisible group $X_0$ over $\F$ and let $\Nilp_{W(\F)}$ be the category of $p$-nilpotent $W(\F)$-algebras.
Then we can consider the Rapoport-Zink moduli problem
\[
    \mathcal{M}^{\RZ} \colon \Nilp_{W(\F)} \rightarrow \Set
\]
sending $R \in \Nilp_{W}$ to the set of isomorphism classes of pairs $(X, \rho)$ consisting of a $p$-divisible group $X$ over $R$ and a quasi-isogeny $\rho \colon X_{0, R/pR} \dashrightarrow X_{R/pR}$.

The central representability result in Rapoport-Zink's book is the following:

\begin{theorem}[{Rapoport-Zink \cite[Thm. 2.16]{RZ}}] \label{thm: thm of RZ}
    The functor $\mathcal{M}^{\RZ}$ is representable by a formal scheme that is locally formally of finite type over $\Spf(W(\F))$.
\end{theorem}

\subsection{Result and overview of proof}

Let $\F$ be a perfect field of characteristic $p$ and let $\mathcal{P}_{0}$ be a nilpotent display over $\F$.
Consider the following functor
\[
    \mathcal{M} = \mathcal{M}^{\Zink} \colon \Nilp_{W(\F)}\rightarrow \Set,
\]
which sends $R\in \Nilp_{W(\F)}$ to the set of isomorphism classes of pairs $(\mathcal{P}, \rho)$, where $\mathcal{P}$ is an (automatically nilpotent) display over $R$ and $\rho \colon \mathcal{P}_{R/p} \dashrightarrow \mathcal{P}_{0, R/p}$ is a quasi-isogeny.

\begin{theorem}[\ref{thm:main-result}] \label{thm: Main theorem}
    The functor $\mathcal{M}$ is representable by a formal scheme that is locally formally of finite type, ind-proper and formally smooth over $\Spf(W(\F))$.
\end{theorem}

\begin{remark}
    Assume that $X_0$ is the formal $p$-divisible group over $\F$ corresponding to $\calP_0$.
    Then by \cite{ZinkDisplay} and \cite{LauInventiones} we have 
    \[
        \mathcal{M}^{\RZ} \simeq \mathcal{M}^{\Zink} \, .
    \]
    The representability of $\mathcal{M}^{\RZ}$ in the case when $X_0$ is formal thus follows from \Cref{thm: Main theorem}.
    In the original statement \cite[Thm. 2.16]{RZ} of Rapoport-Zink they assume that the iso-crystal of $X_0$ is decent.
    We do not need this assumption and our arguments also imply that this assumption is not necessary in their statement.

    Using a result of Fargues \cite[Prop. 5.5]{FarguesMinimality} that expresses the Rapoport-Zink space for a $p$-divisible group $X_0$ via the universal object over the Rapoport-Zink space for the formal part $X_0^{\circ}$, one may deduce \Cref{thm: thm of RZ} from \Cref{thm: Main theorem} in general.
\end{remark}

Our argument has roughly the same structure as the original proof given by Rapoport-Zink, but the basic source of representability is fairly different.

\subsubsection*{Recollection on Rapoport-Zink's proof}

Fix a lift $\widetilde{X}_0$ of $X_0$ to a $p$-divisible group over $W(\F)$.
Then we can define the closed subfunctors
\[
    \mathcal{M}^{\RZ, r} \subseteq \calM^{\RZ}
\]
that are given by the condition that $p^{r} \rho$ lifts to an isogeny $\widetilde{X}_{0, R} \to X$.
These $\calM^{\RZ, r}$ decompose as disjoint unions of further subfunctors $\mathcal{M}^{\RZ, r,s}$, where one also fixes the height of the isogeny $p^{r} \rho$ to be $s$.

Now the $\calM^{\RZ, r, s}$ can be shown to be $p$-adic completions of closed subschemes of Grassmannians:
Giving a point $(X, \rho) \in \calM^{r, s}(R)$ is equivalent of giving the kernel $\ker(p^r \rho)$ and this kernel is a finite locally free subgroup scheme of the $p^s$-torsion $X_0[s]_R$.
In other words, the information about a point in $\calM^{\RZ, r, s}$ is already stored inside $X_0[s]$ and this is the basic reason why the spaces $\mathcal{M}^{\RZ, r,s}$ will have nice finiteness properties.

Using a boundedness result on the Bruhat-Tits building \cite[Prop. 2.18]{RZ} (where their assumption that the iso-crystal of $X_0$ is decent is crucial, cf.\ \cite[Remark 2.20(i)]{RZ}) Rapoport and Zink then prove that $\calM^{\RZ}$ is a formal scheme locally formally of finite type.
This reduction step is the hardest and most technical part of the proof.

\subsubsection*{Overview of our proof}

Similarly as above we start by fixing a lift $\widetilde{\calP}_0$ of $\calP_0$ to a displays over $W(\F)$.
Then one can define subfunctors $\calM^r$ and $\calM^{r, s}$ similarly as above.

Now while the reduction step works similarly as in the original proof, the representability of $\calM^{r, s}$ is not at all clear.
The problem is for example that an isogeny of nilpotent displays
\[
    f\colon \mathcal{P} \rightarrow \widetilde{\mathcal{P}}_{0, R}
\]
might fail to be injective on the level of underlying modules when $W(R)$ contains $p$-torsion and even if $R$ is reduced, so that $f$ is determined by $\coker(f)$, it is not clear that this data descends along filtered colimits:
$W(R)/p^s$ is bigger than $W_s(R)$ and $\coker(f)$ is only a $W(R)/p^s$-module (and not a very well-behaved one).
Therefore it seems difficult to directly adapt the argument involving $p$-divisible groups.

The idea is instead to apply Artin's criterion for representability to the functor $\mathcal{M}^r$ and thereby show that these are $p$-adic formal algebraic spaces locally of finite type over $\Spf(W(\F))$.
To finally deduce that $\calM^r$ is indeed a $p$-adic formal scheme one can then use that the perfection of its special fiber is a perfect scheme by a result of Bhatt-Scholze \cite{BhattScholzeWitt}.

As alluded to above, the most difficult step in the application of Artin's criterion is to show that the functor under consideration is locally of finite presentation, i.e.\ commutes with filtered colimits; this is a priori not clear since one is working with finite projective modules over rings of Witt vectors and Witt vectors do not commute with filtered colimits.
To address this difficulty, it is maybe not too surprising that we make use of truncated displays (as developed by Lau \cite{LauInventiones}, \cite{LauSmoothness} and Lau-Zink \cite{LauZinkTruncated}), since these will then be certain finite projective modules over rings of truncated Witt vectors and therefore here one certainly does not have the problem of commuting colimits.

The passage to truncated displays a priori looses a lot of information and the idea is to show that if one truncates high enough, a sufficient amount of information is already stored in the truncation.
Here a critical inspiration was the homomorphism version of Vasiu's crystalline boundedness principle \cite[Theorem 5.1.1 ]{VasiuBoundedness} which says that over an algebraically closed field of characteristic $p$ morphisms between truncated Dieudonné-modules of the same height lift to morphisms of Dieudonné-modules, provided the level of truncation is higher than a uniform constant.

The key technical statement that we use to show the desired finiteness property of $\mathcal{M}^{r}$ is \Cref{lem:extending-adjoint-matrix}, which roughly says that isogenies of $d$-truncated displays
\[
    f \colon \mathcal{P}_{d} \rightarrow \widetilde{\mathcal{P}}[d]_{0,R}
\]
extend fpqc-locally on $\Spec(R)$ to isogenies of non-truncated displays provided $d$ is suffiently big.
We first solve this problem by hand in characteristic $p$ and then use Grothendieck-Messing deformation theory to deal with the general case.

The verification of the other criteria in Artin's result is a rather straightforward application of the deformation theory of displays.

As a crucial technical input for our strategy, we also prove some foundational results on the behaviour of isogenies and quasi-isogenies of nilpotent displays independently of the theory of $p$-divisible groups.
Namely, we show that homomorphisms between nilpotent displays are $p$-torsionfree, see \Cref{prop:p-torsionfree}, and that the locus where a quasi-isogeny of nilpotent displays is an isogeny is cut out by finitely many equations, see \Cref{cor:locus-isogeny}.

Another result worth mentioning is that we prove the ampleness of the Hodge bundle on $\calM$, after restricting to any quasi-compact closed subscheme, see \Cref{prop: Local ampelness}.
This is a $p$-adic analogue of a result of Moret-Bailly \cite{Moret-BaillyPinceaux}.

\subsection{Further directions}

This text should be read as a proof of concept; we hope that our approach can be generalized to the setting of $\mathcal{G}$-$\mu$-displays.
At least in Hodge-type cases this seems very possible.
Even beyond that, several of our steps very likely generalize to $\mathcal{G}$-$\mu$-displays for arbitrary reductive $\mathcal{G}$.
However there is (at least) one key question that still has to be overcome in that generality:
How does one define the analogue of the subfunctor $\mathcal{M}^{r}$?

In another direction, it would be interesting to see whether one could run a similar argument with some version of prismatic objects to make the approach for displays that we use here also work in that setup to define and construct an appropriate moduli problem. 

We hope to address both questions in the future.

\subsection{Acknowledgments}

S.B. wants to thank Oliver Bültel, Laurent Fargues, Ulrich Görtz, Eike Lau and George Pappas for exchanges about these ideas and their interest.
He thanks especially Oliver Bültel for feedback on previous versions of this text and having suggested a different argument for \Cref{cor:locus-isogeny}, which worked in the Noetherian case.
He also heartily thanks Thomas Zink who once mentioned the crystalline boundedness principle some time ago.
M.H.\ thanks Eike Lau for very useful discussions and Ulrich Görtz for proofreading an earlier version of this article.

While working on this manuscript S.B. was part of the DFG RTG 2553 and M.H.\ was funded by the  SFB-TRR 358 - 491392403.

    \section{Generalities on displays and truncated displays}

        \subsection{Frames and windows}
        In this section we introduce notions of frames and windows over such frames.
We warn the reader that these notions differ from the one given in \cite[Definition 2.1 and Definition 2.3]{LauFrames}; the main difference is that the ideal $I \subseteq S$ in Lau's definition is replaced by an $S$-module $I$ together with a homomorphism $I \to S$.
Its purpose is to conveniently and uniformly phrase the definitions of displays and truncated displays.

\begin{definition}
    A \emph{ring with filtration} is a tuple $(\calS, \calI)$ that is given as follows:
    \begin{itemize}
        \item
        $\calS$ is a ring.

        \item
        $\calI$ is an $\calS$-module that is equipped with an $\calS$-linear map $\calI \to \calS$ whose image is contained in the Jacobson radical $\rad(\calS)$.
    \end{itemize}

    Now fix a ring with filtration $(\calS, \calI)$.
    Then a \emph{finite projective $(\calS, \calI)$-module} is a tuple $\calP = (T, L)$ consisting of finite projective $\calS$-modules $T$ and $L$.
    The \emph{dimension} of $\calP$ is the rank of $T$, its \emph{codimension} is the rank of $L$ and its \emph{height} is the sum of both ranks.
    Given two finite projective $(\calS, \calI)$-modules $\calP = (T, L)$ and $\calP' = (T', L')$ we define
    \[
        \Hom_{(\calS, \calI)}(\calP, \calP') \coloneqq
        \set[\Bigg]{
            \begin{bmatrix}
                a & b \\ c & d
            \end{bmatrix}
        }
        {
        \begin{gathered}
            a \colon T \to T', \quad
            b \colon L \to \calI \otimes_{\calS} T',
            \\
            c \colon T \to L', \quad
            d \colon L \to L'
        \end{gathered}
        } \, .
    \]
\end{definition}

\begin{remark} \label{rmk:s-i-modules}
    \begin{itemize}
        \item
        The condition that the image of $\calI \to \calS$ is contained in $\rad(\calS)$ guarantees that the dimension and the codimension of an $(\calS, \calI)$-module is well-defined, i.e.\ invariant under isomorphism.

        \item
        We have a natural forgetful functor
        \[
        \begin{array}{r c l}
            \curlybr[\big]{\text{finite projective $(\calS, \calI)$-modules}}
            & \longrightarrow
            & \curlybr[\big]{\text{finite projective $\calS$-modules}},
            \\[1ex]
            \calP = (T, L)
            & \longmapsto
            & \calS \otimes_{(\calS, \calI)} \calP \coloneqq T \oplus L \, .
        \end{array}
        \]

        \item
        We try to forget that an $(\calS, \calI)$-module $\calP$ is given by a tuple $(T, L)$.
        Instead we say that a \emph{normal decomposition} of $\calP$ is an isomorphism of $(\calS, \calI)$-modules $\calP \simeq (T, L)$ for some finite projective $\calS$-modules $T$ and $L$, but there are typically many choices for such a normal decomposition and no canonical one.
    \end{itemize}
\end{remark}

\begin{example}
    Let $R$ be a ring and consider the ring with filtration $(R, 0)$.
    Then the category of finite projective $(R, 0)$-modules is equivalent to the category
    \[
        \FilVect(R)
    \]
    of filtered finite projective $R$-modules.
    Objects in this category are finite projective $R$-modules that are equipped with a direct summand and morphisms are homomorphisms of $R$-modules that are compatible with these direct summands.

    When $B \to A$ is a thickening, i.e.\ a surjective ring homomorphism with nilpotent kernel $\fraka \subseteq B$, then we can also consider the ring with filtration $(B, \fraka)$.
    The category of finite projective $(B, \fraka)$-modules then is equivalent to the category
    \[
        \FilVect(B/A)
    \]
    of finite projective $B$-modules whose base change to $A$ is equipped with a direct summand.
\end{example}

\begin{definition}
    A \emph{frame} is a tuple $(\calS, \calI, S, \sigma, \sigma^{\divd})$ that is given as follows:
    \begin{itemize}
        \item
        $(\calS, \calI)$ is a ring with filtration.

        \item
        $S$ is a quotient of $\calS$.

        \item
        $\sigma \colon \calS \to S$ is a ring homomorphism.

        \item
        $\sigma^{\divd} \colon \calI \to S$ is a $\sigma$-linear map such that $p \cdot \sigma^{\divd}(x) = \sigma(x)$ for all $x \in \calI$.
    \end{itemize}
\end{definition}

In the following $\calF = (\calS, \calI, S, \sigma, \sigma^{\divd})$ always denotes a frame.
Given an $\calS$-module $M$ we write $M^{\sigma} = S \otimes_{\sigma, \calS} M$ for its base change along $\sigma$.
Given a homomorphism of $\calS$-modules $f \colon M \to N$ we similarly write $f^{\sigma} \colon M^{\sigma} \to N^{\sigma}$ for the induced homomorphism.

Finally, given finite projective $\calS$-modules $M$ and $N$ and an $\calS$-linear map $f \colon M \to \calI \otimes_{\calS} N$ we write $f^{\sigma, \divd}$ for the composition
\[
    f^{\sigma, \divd} \colon M^{\sigma} \xrightarrow{f^{\sigma}} \calI^{\sigma} \otimes_S N^{\sigma} \xrightarrow{\sigma^{\divd, \sharp} \otimes \id} N^{\sigma},
\]
where $\sigma^{\divd, \sharp} \colon \calI^{\sigma} \to S$ denotes the linearization of $\sigma^{\divd}$.

\begin{lemma} \label{lem:modification}
    There is a natural functor
    \begin{gather*}
        \curlybr[\big]{\text{finite projective $(\calS, \calI)$-modules}} \to \curlybr[\big]{\text{finite projective $S$-modules}},
        \\
        \calP = (T, L) \mapsto \calP^{\sigma} \coloneqq T^{\sigma} \oplus L^{\sigma}, \qquad
        \begin{bmatrix}
            a & b \\ c & d
        \end{bmatrix}
        \mapsto
        \begin{bmatrix}
            a^{\sigma} & b^{\sigma, \divd} \\ p \cdot c^{\sigma} & d^{\sigma}
        \end{bmatrix} \, .
    \end{gather*}
\end{lemma}

\begin{proof}
    This is a direct computation.
\end{proof}

\begin{definition}
    A \emph{window over $\calF$} is a finite projective $(\calS, \calI)$-module $\calP$ that is equipped with an isomorphism of $S$-modules $\Psi \colon \calP^{\sigma} \to S \otimes_{(\calS, \calI)} \calP$.

    A morphism $\calP \to \calP'$ of windows over $\calF$ is a morphism of the underlying finite projective $(\calS, \calI)$-modules such that the diagram
    \[
    \begin{tikzcd}
        \calP^{\sigma} \ar[r] \ar[d, "\Psi"]
        & \calP'^{\sigma} \ar[d, "\Psi'"]
        \\
        S \otimes_{(\calS, \calI)} \calP \ar[r]
        & S \otimes_{(\calS, \calI)} \calP'
    \end{tikzcd}
    \]
    is commutative.
\end{definition}

\begin{remark}
    Let $\calP$ be a window over $\calF$ and suppose that it has a free normal decomposition $(T, L)$.
    After choosing bases of $T$ and $L$ the window $\calP$ is then represented by an invertible matrix
    \[
        U =
        \begin{bmatrix}
            W & X \\
            Y & Z
        \end{bmatrix}
        \in \GL_h(S),
    \]
    where $h$ is the height of $\calP$.

    Given two windows $\calP$ and $\calP'$ over $\calF$ that are represented by matrices $U \in \GL_h(S)$ and $U' \in \GL_{h'}(S)$ a morphism $\calP \to \calP'$ is then represented by a matrix
    \[
        M =
        \begin{bmatrix}
            A & B \\
            C & D
        \end{bmatrix}
        \in \Mat_{h' \times h}(\calS)_{\mu} \coloneqq
        \begin{bmatrix}
            \Mat_{d' \times d}(\calS)
            & \Mat_{d' \times c}(\calI)
            \\
            \Mat_{c' \times d}(\calS)
            & \Mat_{c' \times c}(\calS)
        \end{bmatrix},
    \]
    where $d$, $d'$ and $c$, $c'$ are the dimensions respectively the codimensions of $\calP$ and $\calP'$, that satisfies the equation
    \[
        M \cdot U = U' \cdot \Phi(M) \in \Mat_{h' \times h}(S),
    \]
    where
    \[
        \Phi(M) \coloneqq
        \begin{bmatrix}
            \sigma(A) & \sigma^{\divd}(B) \\
            p \cdot \sigma(C) & \sigma(D)
        \end{bmatrix}
        \in \Mat_{h' \times h}(S) \, .
    \]
\end{remark}

        \subsection{Displays}
        Here we recall the definitions of displays and truncated displays from \cite{ZinkDisplay} and \cite{LauZinkTruncated}.

\medskip

Let $R$ be a $p$-nilpotent ring.
We then have the ring $W(R)$ of ($p$-typical) Witt vectors of $R$ that is equipped with the Frobenius and the Verschiebung.
We write $I_R \subseteq W(R)$ for the kernel of the natural projection $W(R) \to R$ and $F^{\divd} \colon I_R \to W(R)$ for the inverse of the Verschiebung.
These data together give the \emph{Witt frame}
\[
    \calF^{\Witt}_R \coloneqq \roundbr[\big]{W(R), I_R, W(R), F, F^{\divd}} \, .
\]
Similarly write $W_n(R)$ for the ring of $n$-truncated Witt vectors and $I_{n, R} \subseteq W_n(R)$ for its augmentation ideal, and also define
\[
    \calW_n(R) \coloneqq W_{n + 1}(R)/\roundbr[\big]{0, \dotsc, 0, R[p]},
\]
where $R[p] \subseteq R$ denotes the $p$-torsion in $R$.
Then the Frobenius on $W(R)$ induces a ring homomorphism $F \colon \calW_n(R) \to W_n(R)$ and the $W_{n + 1}(R)$-module structure on $I_{n + 1, R}$ factors through $\calW_n(R)$, so that we obtain the \emph{$n$-truncated Witt frame}
\[
    \calF^{\Witt, n}_R \coloneqq \roundbr[\big]{\calW_n(R), I_{n + 1, R}, W_n(R), F, F^{\divd}} \, .
\]

Giving a \emph{(not necessarily nilpotent) display} over $R$ in the sense of \cite{ZinkDisplay} is the same as giving a window over $\calF^{\Witt}_R$.
Similarly, giving an \emph{$n$-truncated display} over $R$ in the sense of \cite{LauZinkTruncated} is the same as giving a window over $\calF^{\Witt, n}_R$.
We write $\Disp(R)$ respectively $\Disp_n(R)$ for the category of displays respectively $n$-truncated displays over $R$.

\begin{remark}
    We tried to give a computationally practicable definition for (truncated) displays.

    The downside of this approach is that our finite projective $(\calS, \calI)$-modules always come with a chosen normal decomposition, as was explained in \Cref{rmk:s-i-modules}, but on the other hand we avoid using the definition of \emph{truncated pairs} given in \cite[Definition 3.2]{LauSmoothness} that might appear rather involved.

    Phrasing the definition of a (truncated) display in terms of the Frobenius-twisted modification $\calP \mapsto \calP^{\sigma}$ from \Cref{lem:modification} is similar to \cite[Definition 2.9]{hoff-ekor-siegel} and is inspired by \cite[Lemma 3.1.5]{kisin-pappas}.
\end{remark}

Now let $B \to A$ be a pd-thickening of $p$-nilpotent rings with kernel $\fraka \subseteq B$.
Write $I_{B/A} \subseteq W(B)$ for the kernel of the projection $W(B) \to A$ and similarly write $I_{n, B/A}$ for the kernel of $W_n(B) \to A$.
The given pd-structure on $\fraka$ induces extensions of $F^{\divd}$ to $F$-linear homomorphisms
\[
    F^{\divd} \colon I_{B/A} \to W(B)
    \quad \text{and} \quad
    F^{\divd} \colon I_{n + 1, B/A} \to W_n(B),
\]
see \cite[Section 2]{LauZinkTruncated}.
Thus we can define the \emph{relative Witt frame}
\[
    \calF^{\Witt}_{B/A} \coloneqq \roundbr[\big]{W(B), I_{B/A}, W(B), F, F^{\divd}}
\]
and the \emph{$n$-truncated relative Witt frame}
\[
    \calF^{\Witt, n}_{B/A} \coloneqq \roundbr[\big]{\calW_n(B), I_{n + 1, B/A}, W_n(B), F, F^{\divd}}.
\]
A \emph{relative display} for $B/A$ is a window over $\calF^{\Witt}_{B/A}$ and an \emph{$n$-truncated relative display} for $B/A$ is a window over $\calF^{\Witt, n}_{B/A}$.
We write $\Disp(B/A)$ and $\Disp_n(B/A)$ for the categories of relative displays and $n$-truncated relative displays for $B/A$.

\begin{definition}
    Let $\calP$ be a display over $R$, choose a normal decomposition $(T, L)$ and write
    \[
        \Psi^{-1} =
        \begin{bmatrix}
            \breve{w} & \breve{x} \\
            \breve{y} & \breve{z}
        \end{bmatrix}
        \colon T \oplus L \to T^F \oplus L^F \, .
    \]
    Then we say that $\calP$ is \emph{nilpotent of nilpotence order less or equal than $e$} if the composition
    \[
        \breve{z}^{(p^{e - 1})} \circ \dotsb \circ \breve{z} \colon (R/pR) \otimes_{W(R)} L \to \roundbr[\Big]{(R/pR) \otimes_{W(R)} L}^{(p^e)}
    \]
    vanishes (this is independent of the chosen normal decomposition), see \cite[Definition 1.11]{LauZinkTruncated}.
    Similarly we also define the nilpotence order of an $n$-truncated display.

    Let $\calP$ be an ($n$-truncated) relative display for $B/A$.
    Then we say that $\calP$ is \emph{nilpotent of nilpotence order less or equal than $e$} if the associated ($n$-truncated) display $\calP_A$ over $A$ has this property.

    We write
    \[
        \Disp^{(e)}(R), \qquad
        \Disp_n^{(e)}(R), \qquad
        \Disp^{(e)}(B/A), \qquad
        \Disp_n^{(e)}(B/A)
    \]
    for the category of displays that are nilpotent of nilpotence order less or equal than $e$ and its truncated and relative variants.
    Similarly we write
    \[
        \Disp^{\nilp}(R), \qquad
        \Disp^{\nilp}(B/A)
    \]
    for the category of (relative) displays that are nilpotent of some nilpotence order.
\end{definition}

Later we will need the following technical statement.

\begin{lemma}\label{lem: abgeschnittene Displays ueber Kolimes sind Kolimes}
    Let $R=\colim_{i\in I}R_{i}$ be a filtered colimit of $p$-nilpotent rings.
    Then the natural functor
    \[
        \colim_{i \in I} \Disp_n(R_i) \to \Disp_n(R)
    \]
    is an equivalence.
    More precisely, this means the following:
    \begin{itemize}
        \item
        Let $\mathcal{P}$ be an $n$-truncated display over $R$.
        Then there exists $i \in I$ and an $n$-truncated display $\mathcal{P}_i$ over $R_i$ such that $\mathcal{P}_{i, R} \simeq \mathcal{P}$.
        
        \item
        Let $\calP$ and $\calP'$ be $n$-truncated displays over $R_i$, for some $i \in I$, and let $f \colon \calP_R \to \calP'_R$ be a morphism of displays over $R$.
        Then there exists $j \in I$ with $j \geq i$ and a morphism $f_j \colon \calP_{R_j} \to \calP'_{R_j}$ such that we have $f_{j, R} = f$.
        
        \item
        Let $f, g \colon \calP \to \calP'$ be morphisms of $n$-truncated displays over $R_i$ for some $i \in I$ such that $f_R = g_R$.
        Then there exists $j \in I$ with $j \geq i$ such that $f_{R_j} = g_{R_j}$.
    \end{itemize}
\end{lemma}

\begin{proof}
    This follows formally from the fact that $\calF_{\calW_n(R)} = \colim_{i \in I} \calF_{\calW_{n}(R_i)}$ and \cite[Tag 05LI and Tag 05N7]{stacks}.
\end{proof}

        \subsection{Deformation theory}
        Next we recall the deformation theory of nilpotent (truncated) displays.

\begin{lemma} \label{prop:deformation-theory}
    Let $B \to A$ be a pd-thickening of $p$-nilpotent rings with kernel $\fraka \subseteq B$.
    Then the natural diagrams
    \[
    \begin{tikzcd}
        \Disp(B) \ar[r] \ar[d]
        & \Disp(B/A) \ar[d]
        \\
        \FilVect(B) \ar[r]
        & \FilVect(B/A)
    \end{tikzcd}
    \quad \text{and} \quad
    \begin{tikzcd}
        \Disp_n(B) \ar[r] \ar[d]
        & \Disp_n(B/A) \ar[d]
        \\
        \FilVect(B) \ar[r]
        & \FilVect(B/A)
    \end{tikzcd}
    \]
    are $2$-Cartesian; here the vertical functors are induced by the morphism of rings with filtrations
    \[
        \roundbr[\big]{W(B), I_B} \to (B, 0), \qquad \roundbr[\big]{W(B), I_{B/A}} \to (B, \fraka)
    \]
    and the obvious truncated analogues.
\end{lemma}

\begin{proof}
    This follows immediately from the definitions, see for example \cite[Proposition 45]{ZinkDisplay} in the non-truncated case.
\end{proof}

\begin{theorem}[Zink, Lau-Zink] \label{thm:deformation-theory}
    Let $B \to A$ be a pd-thickening of $p$-nilpotent rings with kernel $\fraka$ and let $r$ be a non-negative integer such that $p^r \fraka = 0$.
    Then for a fixed non-negative integer $e$ there exists a constant $\const = \const(r, e)$, depending only on $r$ and $e$, and a natural functor
    \[
        \Disp_{n + \const}^{(e)}(A) \to \Disp_n^{(e)}(B/A),
    \]
    unique up to unique isomorphism, that makes the diagram
    \[
    \begin{tikzcd}
        \Disp_{n + \const}^{(e)}(B/A) \ar[r] \ar[d]
        & \Disp_{n + \const}^{(e)}(A) \ar[d] \ar[ld]
        \\
        \Disp_n^{(e)}(B/A) \ar[r]
        & \Disp_n^{(e)}(A)
    \end{tikzcd}
    \]
    $2$-commutative.
    Consequently the natural functor $\Disp^{\nilp}(B/A) \to \Disp^{\nilp}(A)$ is an equivalence.
\end{theorem}

\begin{proof}
    The first part is a reformulation of \cite[Proposition 2.3]{LauZinkTruncated}.
    More precisely it is shown there that the functor $\Disp_{n + \const}^{(e)}(B/A) \to \Disp_{n + \const}^{(e)}(A)$ is essentially surjective and surjective on morphisms and that for relative $(n + \const)$-truncated displays $\calP_{n + \const}$ and $\calP'_{n + \const}$ for $B/A$ there exists a (unique) factorization
    \[
    \begin{tikzcd}
        \Hom_{B/A} \roundbr[\big]{\calP_{n + \const}, \calP'_{n + \const}} \ar[r] \ar[d]
        & \Hom_{B/A} \roundbr[\big]{\calP_{n + \const}[n], \calP'_{n + \const}[n]}
        \\
        \Hom_A \roundbr[\big]{\calP_{n + \const, A}, \calP'_{n + \const, A}} \ar[ru]
    \end{tikzcd} \, .
    \]
    From this our claim follows formally.

    The second part was originally shown in \cite[Theorem 44]{ZinkDisplay} but also follows from the first part by passing to the limit for $n \to \infty$.
\end{proof}

\begin{corollary}\label{cor: Corollary Def statement via truncated displays}
    Let $B \to A$ be a pd-thickening of $p$-nilpotent rings and let $r$ be a non-negative integer such that $p^r \ker(B \to A) = 0$.
    Fix a non-negative integer $e$ and let $\const = \const(r, e)$ be the constant from \Cref{thm:deformation-theory}.
    Now suppose we are given the following data.
    \begin{itemize}
        \item
        A display $\overline{\calP}$ over $A$ and a lift $\calP_{n + \const}$ of $\overline{\calP}[n + \const]$ to an $(n + \const)$-truncated display over $B$.

        \item
        A display $\calP'$ over $B$.
        
        \item
        Morphisms $\overline{f} \colon \overline{\calP} \to \calP'_A$ and $f_{n + \const} \colon \calP_{n + \const} \to \calP'[n+\const]$ such that $\overline{f}[n + \const] = f_{n + \const, A}$.
    \end{itemize}
    Assume that $\overline{\calP}$, $\calP_{n + \const}$ and $\calP'$ are all nilpotent of nilpotency order less or equal than $e$.

    Then there exists a display $\calP$ over $B$ that compatibly lifts $\overline{\calP}$ and $\calP_{n + \const}[n]$ and a morphism $f \colon \calP \to \calP'$ such that $f_A = \overline{f}$ and $f[n] = f_{n + \const}[n]$.
\end{corollary}

\begin{proof}
    By \Cref{thm:deformation-theory} we have a $2$-commutative diagram
    \[
    \begin{tikzcd}
        \Disp^{(e)}(B) \ar[r] \ar[d]
        & \Disp^{(e)}(B/A) \times_{\Disp_n^{(e)}(B/A)} \Disp_n^{(e)}(B) \ar[d]
        \\
        \Disp^{(e)}(A) \times_{\Disp_{n + \const}^{(e)}(A)} \Disp_{n + \const}^{(e)}(B) \ar[ru] \ar[r]
        & \Disp^{(e)}(A) \times_{\Disp_n^{(e)}(A)} \Disp_n^{(e)}(B)
    \end{tikzcd} \, .
    \]
    Now we have objects
    \[
        \roundbr[\big]{\overline{\calP}, \calP_{n + \const}} \in \Disp^{(e)}(A) \times_{\Disp_{n + \const}^{(e)}(A)} \Disp_{n + \const}^{(e)}(B)
        \quad \text{and} \quad
        \calP' \in \Disp^{(e)}(B)
    \]
    and a morphism
    \[
        \roundbr[\big]{\overline{f}, f_{n + \const}} \colon \roundbr[\big]{\overline{\calP}, \calP_{n + \const}} \to \roundbr[\big]{\calP'_A, \calP'[n + \const]}
        \quad \text{in} \quad
        \Disp^{(e)}(A) \times_{\Disp_{n + \const}^{(e)}(A)} \Disp_{n + \const}^{(e)}(B) \, .
    \]
    It follows from \Cref{prop:deformation-theory} that the upper horizontal functor is an equivalence of categories.
    Thus there exists a tuple $(\calP, f)$ consisting of a nilpotent display $\calP$ of nilpotentcy order less or equal than $e$ and a morphism $f \colon \calP \to \calP'$, unique up to unique isomorphism, such that its image in the upper right corner of the diagram agrees with the image of the tuple $((\overline{\calP}, \calP_{n + \const}), (\overline{f}, f_{n + \const}))$.
    The $2$-commutativity of the lower right triangle of the diagram then implies the compatibilites $f_A = \overline{f}$ and $f[n] = f_{n + \const}[n]$ as desired.
\end{proof}

        \subsection{Morphisms of displays}
        \label{sec:morphisms}

In this section we prove some (new) foundational results about morphisms of displays.

\begin{lemma} \label{lem:homogeneous-equation}
    Let $R$ be an $\Fp$-algebra and let
    \[
        B \in \Mat_b \roundbr[\big]{W(R)}
        \quad \text{and} \quad
        C \in \Mat_c \roundbr[\big]{W(R)}
    \]
    such that $C_0^{(p^{e - 1})} \dotsm C_0 = 0$ for some non-negative integer $e$ (where $C_0 \in \Mat_c(R)$ denotes the zeroth Witt component of $C$).

    Then the system of equations
    \[
        \curlybr[\Bigg]{
            \begin{aligned}
                V(X) &= F^{e - 1}(B) \cdot X \cdot F^{e - 1}(C), \\
                F(X) &= 0
            \end{aligned}
        }, \qquad
        X \in \Mat_{b \times c} \roundbr[\big]{W(R)}
    \]
    has only the trivial solution $X = 0$.
\end{lemma}

\begin{proof}
    Let $X$ be a solution.
    Then we can recursively substitute the first equation into itself to obtain
    \[
        V^e(X) = B \dotsm F^{e - 1}(B) \cdot X \cdot F^{e - 1}(C) \dotsm C \, .
    \]
    Now our assumption on $C$ implies that we have $F^{e - 1}(C) \dotsm C = V(\adjust{C})$ for some $\adjust{C} \in \Mat_c(W(R))$.
    Hence we get
    \[
        X \cdot F^{e - 1}(C) \dotsm C
        = X \cdot V \roundbr[\big]{\adjust{C}}
        = V \roundbr[\big]{F(X) \cdot \adjust{C}} = 0,
    \]
    so that $V^e(X) = 0$ and consequently $X = 0$ as desired.
\end{proof}

\begin{proposition} \label{prop:p-torsionfree}
    Let $\calP$, $\calP'$ be displays over $R$ and assume that $\calP$ is nilpotent.
    Then $\Hom_R(\calP, \calP')$ is $p$-torsionfree.
\end{proposition}

\begin{proof}
    The functor on $R$-algebras
    \[
        S \mapsto \Hom_S \roundbr[\big]{\calP_S, \calP'_S}[p]
    \]
    is representable by a (pointed) affine $R$-scheme $\Spec(A)$.
    Our goal is to show that $A = R$ or equivalently that $R \to A$ is surjective.

    By Nakayama's Lemma \cite[Tag 00DV]{stacks} we are free to replace $R$ by $R/I$ for a nilpotent ideal $I \subseteq R$.
    Using this we make the following reduction steps.
    \begin{itemize}
        \item
        We may assume that $R$ is of characteristic $p$.

        \item
        After passing to a Zariski cover we may assume that $\calP$ and $\calP'$ have free normal decompositions and thus can be represented by block matrices
        \[
            \begin{bmatrix}
                W & X \\ Y & Z
            \end{bmatrix}
            \quad \text{and} \quad
            \begin{bmatrix}
                W' & X' \\ Y' & Z'
            \end{bmatrix} \, .
        \]
        After replacing $R$ by a faithfully ring extension we may further assume that we can write $\breve{Z} = F^{e - 1}(\adjust{\breve{Z}})$ and $W' = F^{e - 1}(\adjust{W'})$; here $e$ denotes the nilpotency order of $\calP$ and we have set
        \[
            \begin{bmatrix}
                \breve{W} & \breve{X} \\ \breve{Y} & \breve{Z}
            \end{bmatrix}
            \coloneqq
            \begin{bmatrix}
                W & X \\ Y & Z
            \end{bmatrix}^{-1} \, .
        \]
        
        \item
        We now have
        \[
            \roundbr[\Big]{\breve{Z}^{\anglebr{1}, (p^{e - 1})}_0 \dotsm \adjust{\breve{Z}_0}}^{(p^{e - 1})} = \breve{Z}_0^{(p^{e - 1})} \dotsm \breve{Z}_0 = 0
        \]
        so that the matrix entries of $\breve{Z}^{\anglebr{1}, (p^{e - 1})}_0 \dotsm \adjust{\breve{Z}_0}$ generate a nilpotent ideal.
        Thus we may assume that this matrix is already zero itself.
    \end{itemize}

    Now let $S$ be an $R$-algebra and let $f \colon \calP_S \to \calP'_S$ be a morphism such that $p \cdot f = 0$.
    Then $f$ is represented by a block matrix
    \[
        \begin{bmatrix}
            A & V(B) \\ C & D
        \end{bmatrix}
    \]
    that satisfies
    \[
        \begin{bmatrix}
            A & V(B) \\ C & D
        \end{bmatrix}
        =
        \begin{bmatrix}
            W' & X' \\ Y' & Z'
        \end{bmatrix}
        \cdot
        \begin{bmatrix}
            F(A) & B \\ p \cdot F(C) & F(D)
        \end{bmatrix}
        \cdot
        \begin{bmatrix}
            \breve{W} & \breve{X} \\ \breve{Y} & \breve{Z}
        \end{bmatrix}
        \quad \text{and} \quad
        F \roundbr*{
        \begin{bmatrix}
            A & V(B) \\ C & D
        \end{bmatrix}
        }
        = 0 \, .
    \]
    In particular we then have
    \[
        V(B) = F^{e - 1} \roundbr[\big]{\adjust{W'}} \cdot B \cdot F^{e - 1} \roundbr[\big]{\adjust{\breve{Z}}}
        \quad \text{and} \quad
        F(B) = 0
    \]
    so that we may apply \Cref{lem:homogeneous-equation} to obtain $B = 0$.
    From this it immediately follows that also $A, C, D = 0$ so that $f = 0$ as desired.
\end{proof}

\begin{lemma} \label{lem:homogeneous-equation-truncated}
    Let $R$ be an $\Fp$-algebra and let
    \[
        B \in \Mat_b \roundbr[\big]{W_n(R)}
        \quad \text{and} \quad
        C \in \Mat_c \roundbr[\big]{W_n(R)}
    \]
    such that $C_0^{(p^{e - 1})} \dotsm C_0 = 0$ for some non-negative integer $e < n$.

    Then every solution of the system of equations
    \[
        \curlybr[\Bigg]{
            \begin{aligned}
                V(X) &= F^{e - 1}(B) \cdot X \cdot F^{e - 1}(C), \\
                F(X) &= 0
            \end{aligned}
        }, \qquad
        X \in \Mat_{b \times c} \roundbr[\big]{W_n(R)}
    \]
    satisfies $X[n - e] = 0 \in \Mat_{b \times c}(W_{n - e}(R))$.
\end{lemma}

\begin{proof}
    This follows from the same argument as in \Cref{lem:homogeneous-equation}.
\end{proof}

\begin{proposition} \label{prop:p-torsionfree-truncated}
    Let $R$ be a $\ZZ/p^a$-algebra.
    Let $\calP_n$, $\calP'_n$ be nilpotent $n$-truncated displays over $R$ of nilpotence order less or equal than $e$ and let $f \colon \calP_n \to \calP'_n$ be a morphism such that $p \cdot f = 0$.
    Then there exists a constant $\const = \const(a, e, c)$, depending only on $a$, $e$ and the codimension $c$ of $\calP_n$, such that
    \[
        f[n - \const] = 0
    \]
    whenever $n > \const$.
\end{proposition}

\begin{proof}
    By Witt-descent, see \cite[Proposition A.5]{LauZinkTruncated}, we may replace $R$ by a faithfully flat extension.
    By \Cref{thm:deformation-theory} (see also \cite[Corollary 2.4]{LauZinkTruncated}) we may also replace $R$ by $R/I$ for an ideal $I \subseteq R$ that is nilpotent of order bounded in terms of $a$, $e$ and $c$.

    Now the claim follows from the same argument as in \Cref{prop:p-torsionfree}, using \Cref{lem:homogeneous-equation-truncated} instead of \Cref{lem:homogeneous-equation}.
\end{proof}

\begin{lemma} \label{lem:inhomogeneous-equation}
    Let $R$ be an $\Fp$-algebra and let
    \[
        B \in \Mat_b \roundbr[\big]{W(R)}, \quad
        C \in \Mat_c \roundbr[\big]{W(R)}
        \quad \text{and} \quad
        A, Y \in \Mat_{b \times c} \roundbr[\big]{W(R)}
    \]
    such that
    \[ \label{eq:inhomogeneous-equation-3}
        V(Y) = F^{e + 2}(B) \cdot Y \cdot F^{e + 2}(C) + F(A) \tag{$\#$}
    \]
    and $C_0^{(p^{e - 1})} \dotsm C_0 = 0$ for some non-negative integer $e$.
    Write $F^{e - 1}(C) \dotsm C = V(\adjust{C})$ as before.

    Then the system of equations
    \[ \label{eq:inhomogeneous-equation-1}
        \curlybr[\Bigg]{
            \begin{aligned}
                V(X) &= F^{e + 1}(B) \cdot X \cdot F^{e + 1}(C) + A, \\
                F(X) &= Y
            \end{aligned}
        }, \qquad
        X \in \Mat_{b \times c} \roundbr[\big]{W(R)} \tag{$\star$}
    \]
    is equivalent to the single equation
    \[ \label{eq:inhomogeneous-equation-2}
    \begin{aligned}
        V^e(X) &= F^2(B) \dotsm F^{e + 1}(B) \cdot V(Y) \cdot F \roundbr[\big]{\adjust{C}} \\
        & \qquad+ \sum_{i = 0}^{e - 1} F^2(B) \dotsm F^{e - i}(B) \cdot V^i(A) \cdot F^{e - i}(C) \dotsm F^2(C) \, .
    \end{aligned}
    \tag{$\star \star$}
    \]
    In particular the system $(\text{\ref{eq:inhomogeneous-equation-1}})$ has a (unique) solution if and only if the $j$-th Witt component of the right hand side of $(\text{\ref{eq:inhomogeneous-equation-2}})$ vanishes for $j = 0, \dotsc, e - 1$.
\end{lemma}

\begin{proof}
    First let $X$ be a solution to $(\text{\ref{eq:inhomogeneous-equation-1}})$.
    Recursively substituting the first equation into itself yields
    \begin{align*}
        V^e(X) &= F^2(B) \dotsm F^{e + 1}(B) \cdot X \cdot F^{e + 1}(C) \dotsm F^2(C) \\
        & \qquad + \sum_{i = 0}^{e - 1} F^2(B) \dotsm F^{e - i}(B) \cdot V^i(A) \cdot F^{e - i}(C) \dotsm F^2(C) \, .
    \end{align*}
    Now, using the second equation, we get
    \[
        X \cdot F^{e + 1}(C) \dotsm F^2(C) = X \cdot V \roundbr[\big]{F^2 \roundbr[\big]{\adjust{C}}} = V(Y) \cdot F \roundbr[\big]{\adjust{C}}
    \]
    so that we obtain $(\text{\ref{eq:inhomogeneous-equation-2}})$.

    Now let $X$ be a solution to $(\text{\ref{eq:inhomogeneous-equation-2}})$.
    From recursively substituting $(\text{\ref{eq:inhomogeneous-equation-3}})$ into itself we get
    \[
        F \roundbr[\big]{V^e(X)} = V^e(Y)
    \]
    so that $F(X) = Y$.
    Thus it is left to show
    \[
        V^{e + 1}(X) \overset{!}{=} F(B) \cdot V^e(X) \cdot F(C) + V^e(A) \, .
    \]
    Substituting $(\text{\ref{eq:inhomogeneous-equation-2}})$ into both sides yields
    \begin{align*}
        V^{e + 1}(X) &= F(B) \dotsm F^e(B) \cdot V^2(Y) \cdot \adjust{C} \\
        & \qquad + \sum_{i = 0}^{e - 1} F(B) \dotsm F^{e - 1 - i}(B) \cdot V^{i + 1}(A) \cdot F^{e - 1 - i}(C) \dotsm F(C) \\
        &= F(B) \dotsm F^{e + 1}(B) \cdot V(Y) \cdot F^{e + 1}(C) \cdot \adjust{C} \\
        & \qquad + \sum_{i = 0}^e F(B) \dotsm F^{e - i}(B) \cdot V^i(A) \cdot F^{e - i}(C) \dotsm F(C),
    \end{align*}
    where we used that
    \[
        V^2(Y) = F^{e + 1}(B) \cdot V(Y) \cdot F^{e + 1}(C) + p \cdot A
        \quad \text{and} \quad
        p \cdot \adjust{C} = F^e(C) \dotsm F(C),
    \]
    and
    \begin{align*}
        F(B) \cdot V^e(X) \cdot F(C) + V^e(A) &= F(B) \dotsm F^{e + 1}(B) \cdot V(Y) \cdot F \roundbr[\big]{\adjust{C}} \cdot F(C) \\
        & \qquad + \sum_{i = 0}^e F(B) \dotsm F^{e - i}(B) \cdot V^i(A) \cdot F^{e - i}(C) \dotsm F(C) \, .
    \end{align*}
    By noting that
    \[
        V \roundbr[\Big]{F^{e + 1}(C) \cdot \adjust{C}} = F^e(C) \dotsm C = V \roundbr[\Big]{F \roundbr[\big]{\adjust{C}} \cdot F(C)}
    \]
    we can conclude that both terms are equal.
\end{proof}

\begin{lemma} \label{lem:locus-divisible-by-p-deform}
    Let $g \colon \calP \to \calP'$ be a morphism of nilpotent displays over $R$.
    Let $I \subseteq R$ be a finitely generated nilpotent ideal and assume that the functor on $R/I$-algebras
    \[
        S \mapsto
        \begin{cases}
            \curlybr{\ast} & \text{if $g_S \colon \calP_S \to \calP'_S$ is divisible by $p$,}
            \\
            \varnothing & \text{else}
        \end{cases}
    \]
    is representable by $\Spec((R/I)/\overline{K})$ for a finitely generated ideal $\overline{K} \subseteq R/I$.

    Then also the analogous functor on $R$-algebras is representable by $\Spec(R/J)$ for a finitely generated ideal $J \subseteq R$.
\end{lemma}

\begin{proof}
    We may assume without loss of generality that $I^2 = 0$.
    Write $K \subseteq R$ for the preimage of $\overline{K} \subseteq R/I$.
    If $g_S$ is divisible by $p$ then we have
    \[
        \overline{K} \cdot (S/IS) = 0,
        \quad \text{i.e.} \quad
        KS \subseteq IS,
    \]
    and consequently $K^2 S = 0$.
    As $K^2 \subseteq R$ is finitely generated we can replace $R$ by $R/K^2$ and thus may assume that $K^2 = 0$.
    
    Equip $K \subseteq R$ with the trivial pd-structure.
    Then by \Cref{thm:deformation-theory} the morphism of displays $(p^{-1} g)_{R/K}$ lifts uniquely to a morphism of relative displays
    \[
        \roundbr[\big]{p^{-1} \cdot g}_{R/(R/K)} \colon \calP_{R/(R/K)} \to \calP'_{R/(R/K)} \, .
    \]
    The morphism $g_S \colon \calP_S \to \calP'_S$ now is divisible by $p$ if and only if $(p^{-1} g)_{S/(S/KS)}$ is compatible with the Hodge filtrations of $\calP_S$ and $\calP'_S$ by \Cref{prop:deformation-theory}.
    This condition is clearly cut out by a finitely generated ideal $J \subseteq R$ as desired.
\end{proof}

\begin{proposition} \label{prop:locus-divisible-by-p}
    Let $g \colon \calP \to \calP'$ be a morphism of nilpotent displays over $R$.
    Then the functor on $R$-algebras
    \[
        S \mapsto
        \begin{cases}
            \curlybr{\ast} & \text{if $g_S \colon \calP_S \to \calP'_S$ is divisible by $p$,}
            \\
            \varnothing & \text{else}
        \end{cases}
    \]
    is representable by $\Spec(R/J)$ for a finitely generated ideal $J \subseteq R$.
\end{proposition}

\begin{proof}
    We are free to replace $R$ by $R/I$ for a finitely generated nilpotent ideal by \Cref{lem:locus-divisible-by-p-deform}.
    Thus, arguing similarly as in \Cref{prop:p-torsionfree}, we may assume that $R$ is of characteristic $p$, that $\calP$, $\calP'$ and $g$ are represented by block matrices
    \[
        \begin{bmatrix}
            W & X \\ Y & Z
        \end{bmatrix}, \quad
        \begin{bmatrix}
            W' & X' \\ Y' & Z'
        \end{bmatrix}
        \quad \text{and} \quad
        \begin{bmatrix}
            K & V(L) \\ M & N
        \end{bmatrix}
    \]
    and that we can write $\breve{Z} = F^{e + 1}(\adjust{\breve{Z}})$ and $W' = F^{e + 1}(\adjust{W'})$ with $(\adjust{\breve{Z}}_0)^{(p^{e - 1})} \dotsm \adjust{\breve{Z}}_0 = 0$, where we have set
    \[
        \begin{bmatrix}
            \breve{W} & \breve{X} \\ \breve{Y} & \breve{Z}
        \end{bmatrix}
        \coloneqq
        \begin{bmatrix}
            W & X \\ Y & Z
        \end{bmatrix}^{-1}
    \]
    as before.
    If $g_S$ is divisible by $p$ then we certainly have $K_{S, 0}, L_{S, 0}, M_{S, 0}, N_{S, 0} = 0$.
    Hence, after replacing $R$ with its quotient by the finitely generated ideal generated by the matrix entries of $K_0$, $L_0$, $M_0$ and $N_0$, we may assume that we can write
    \[
        K = V \roundbr[\big]{\adjust{K}}, \quad
        L = V \roundbr[\big]{\adjust{L}}, \quad
        M = V \roundbr[\big]{\adjust{M}}
        \quad \text{and} \quad
        N = V \roundbr[\big]{\adjust{N}} \, .
    \]
    Now $g_S$ is divisible by $p$ if and only if the system of equations
    \[
        \curlybr*{
            \vphantom{\begin{gathered} \begin{pmatrix} A \\ A \end{pmatrix} \\ \begin{pmatrix} A \\ A \end{pmatrix} \\ A \end{gathered}}
            \begin{aligned}
                \begin{bmatrix}
                    A & V(B) \\ C & D
                \end{bmatrix}
                &=
                \begin{bmatrix}
                    W' & X' \\ Y' & Z'
                \end{bmatrix}
                \cdot
                \begin{bmatrix}
                    F(A) & B \\ p \cdot F(C) & F(D)
                \end{bmatrix}
                \cdot
                \begin{bmatrix}
                    \breve{W} & \breve{X} \\ \breve{Y} & \breve{Z}
                \end{bmatrix},
                \\[1ex]
                p \cdot
                \begin{bmatrix}
                    A & V(B) \\ C & D
                \end{bmatrix}
                &=
                \begin{bmatrix}
                    K & V(L) \\ M & N
                \end{bmatrix}
            \end{aligned}
        }, \qquad
        (A, B, C, D)
    \]
    has a (unique) solution over $W(S)$, in which case $p^{-1} \cdot g_S$ is represented by the matrix
    \[
        \begin{bmatrix}
            A & V(B) \\ C & D
        \end{bmatrix} \, .
    \]
    This system is equivalent to the system
    \[ \label{eq:locus-divisible-by-p-1}
        \curlybr*{
            \vphantom{\begin{gathered} \begin{pmatrix} A \\ A \end{pmatrix} \\ \begin{pmatrix} A \\ A \end{pmatrix} \end{gathered}}
            \begin{gathered}
                \begin{bmatrix}
                    A & V(B) \\ C & D
                \end{bmatrix}
                =
                \begin{bmatrix}
                    W' & X' \\ Y' & Z'
                \end{bmatrix}
                \cdot
                \begin{bmatrix}
                    \adjust{K} & B \\ p \cdot \adjust{M} & \adjust{N}
                \end{bmatrix}
                \cdot
                \begin{bmatrix}
                    \breve{W} & \breve{X} \\ \breve{Y} & \breve{Z}
                \end{bmatrix},
                \\[1ex]
                F(A) = \adjust{K}, \quad F(B) = \adjust{L}, \quad F(C) = \adjust{M}, \quad F(D) = \adjust{N}
            \end{gathered}
        }, \qquad
        (A, B, C, D) \tag{$\dagger$}
    \]
    and we claim that it is also equivalent to the system
    \[ \label{eq:locus-divisible-by-p-2}
        \curlybr*{
            \vphantom{\begin{gathered} \begin{pmatrix} A \\ A \end{pmatrix} \\ A \end{gathered}}
            \begin{aligned}
                V(B) &= F^{e + 1} \roundbr[\big]{\adjust{W'}} \cdot B \cdot F^{e + 1} \roundbr[\big]{\adjust{\breve{Z}}} + \roundbr[\Big]{W' \adjust{K} \breve{X} + p \cdot X' \adjust{M} \breve{X} + X' \adjust{N} \breve{Z}}, \\
                F(B) &= \adjust{L}
            \end{aligned}
        },
        \quad B \, . \tag{$\dagger \dagger$}
    \]
    Assuming this we can apply \Cref{lem:inhomogeneous-equation} to conclude, using that $g$ is a morphism of displays in order to verify the condition $(\text{\ref{eq:inhomogeneous-equation-3}})$ appearing in that lemma.

    Thus it remains to show that the two labeled systems above really are equivalent.
    Note that the first equation in $(\text{\ref{eq:locus-divisible-by-p-2}})$ is precisely the upper right matrix entry of the first equation $(\text{\ref{eq:locus-divisible-by-p-1}})$, so that for any solution $(A, B, C, D)$ of $(\text{\ref{eq:locus-divisible-by-p-1}})$ the matrix $B$ is a solution of $(\text{\ref{eq:locus-divisible-by-p-2}})$.
    Now suppose conversely that $B$ is a solution of $(\text{\ref{eq:locus-divisible-by-p-2}})$.
    Then there clearly exist uniquely determined matrices $A$, $C$ and $D$ such that $(A, B, C, D)$ solves the first equation in $(\text{\ref{eq:locus-divisible-by-p-1}})$.
    The remaining equations then also are satisfied, as can be seen by multiplying the first equation by $p$ and using that $g$ is a morphism of displays.
\end{proof}

\begin{proposition} \label{prop:divisible-by-p-truncated}
    Let $R$ be a $\ZZ/p^a$-algebra, let $\calP$ and $\calP'$ be nilpotent displays over $R$ of nilpotence order less or equal than $e$ and let $f \colon \calP \to \calP'$ be a morphism.
    Then there exists a positive integer $n = n(a, e, c)$, depending only on $a$, $e$ and the codimension $c$ of $\calP$, such that the following conditions are equivalent:
    \begin{enumerate}
        \item
        The morphism $f \colon \calP \to \calP'$ is divisible by $p$.

        \item
        The truncated morphism $f[n] \colon \calP[n] \to \calP'[n]$ is divisible by $p$.
    \end{enumerate}
\end{proposition}

\begin{proof}
    \enquote{$(1) \Longrightarrow (2)$}:
    This is clear.

    \enquote{$(2) \Longrightarrow (1)$}:
    Choose $n > \const_1 + \const_2$ where $\const_1 = \const_1(a, e, c)$ and $\const_2 = \const_2(a, e)$ are the constants from \Cref{prop:p-torsionfree-truncated} and \Cref{thm:deformation-theory}.
    Now suppose that there exists a morphism of $n$-truncated displays
    \[
        \adjust{f}_n \colon \calP[n] \to \calP'[n]
    \]
    with $f[n] = p \cdot \adjust{f}_n$.
    We already know by \Cref{prop:locus-divisible-by-p} that the locus in $\Spec(R)$ where $f$ is divisible by $p$ is cut out by a finitely generated ideal $I \subseteq R$.
    The base change of $f$ to $(R/p)^{\perf}$ is clearly divisible by $p$ because its $1$-truncation vanishes by assumption.
    Thus $I$ is in fact a nilpotent ideal and we may assume without loss of generality that $I^2 = 0$.

    Equip $I \subseteq R$ with the trivial pd-structure.
    Then $(p^{-1} \cdot f)_{R/I}$ lifts uniquely to a morphism of relative displays
    \[
        \roundbr[\big]{p^{-1} \cdot f}_{R/(R/I)} \colon \calP_{R/(R/I)} \to \calP'_{R/(R/I)}
    \]
    and our goal is to show that the induced homomorphism on underlying finite projective $R$-modules is compatible with the Hodge filtrations of $\calP$ and $\calP'$, see \Cref{prop:deformation-theory} and \Cref{thm:deformation-theory}.
    Note that this can be checked after passing to the $1$-truncation.

    Now we have
    \[
        p \cdot \roundbr[\Big]{\roundbr[\big]{p^{-1} \cdot f}_{R/I}[n] - \roundbr[\big]{\adjust{f}_n}_{R/I}} = 0,
        \quad \text{so that} \quad
        \roundbr[\big]{p^{-1} \cdot f}_{R/I}[n - \const_1] = \roundbr[\big]{\adjust{f}_n}_{R/I}[n - \const_1]
    \]
    by \Cref{prop:p-torsionfree-truncated}.
    Applying \Cref{thm:deformation-theory} we thus obtain
    \[
        \roundbr[\big]{p^{-1} \cdot f}_{R/(R/I)}[n - \const_1 - \const_2] = \roundbr[\big]{\adjust{f}_n}_{R/(R/I)}[n - \const_1 - \const_2] \, .
    \]
    As $(\adjust{f}_n)_{R/(R/I)}[n - \const_1 - \const_2]$ is the base change of the morphism $\adjust{f}_n[n - \const_1 - \const_2]$ of $(n - \const_1 - \const_2)$-truncated displays over $R$ it is compatible with Hodge filtrations, finishing the proof.
\end{proof}

        \subsection{Isogenies and quasi-isogenies}
        We discuss the notions of isogenies and quasi-isogenies of displays, roughly following \cite[Chapter 17.6]{ZinkSFB}.

\begin{lemma}
    Let $\calF = (\calS, \calI, S, \sigma, \sigma^{\divd})$ be a frame.
    Then, for a finite projective $(\calS, \calI)$-module $\calP$ of fixed height and dimension, there exists a natural isomorphism
    \[
        \det \roundbr[\big]{\calS \otimes_{(\calS, \calI)} \calP}^{\sigma} \to \det \roundbr[\big]{\calP^{\sigma}} \, .
    \]
    Consequently we obtain a natural functor
    \[
        \det \colon \curlybr[\Bigg]{\begin{gathered} \text{windows $\calP$ over $\calF$} \\ \text{of height $h$ and dimension $d$} \end{gathered}}
        \to
        \set[\Bigg]{(\calL, \Psi)}{\begin{gathered} \text{$\calL$ an invertible $\calS$-module,} \\ \text{$\Psi \colon \calL^{\sigma} \to S \otimes_{\calS} \calL$ an isomorphism} \end{gathered}} \, .
    \]
\end{lemma}

\begin{proof}
    Choose a normal decomposition $(T, L)$ of $\calP$.
    Then we have
    \[
        \det \roundbr[\big]{\calS \otimes_{(\calS, \calI)} \calP}^{\sigma} \simeq \det \roundbr[\big]{T \oplus L}^{\sigma}
        \quad \text{and} \quad
        \det \roundbr[\big]{\calP^{\sigma}} \simeq \det \roundbr[\big]{T^{\sigma} \oplus L^{\sigma}}
    \]
    so that we have an obvious candidate for the desired isomorphism.
    To check naturality we have to see that given a morphism $\calP \to \calP'$ of finite projective $(\calS, \calI)$-modules of the same height and dimension that is represented by a matrix
    \[
        \begin{bmatrix}
            a & b \\ c & d
        \end{bmatrix}
    \]
    with respect to a choice of normal decompositions of $\calP$ and $\calP'$ we have the equality
    \[
        \det \roundbr[\Bigg]{
            \begin{bmatrix}
                a & b \\ c & d
            \end{bmatrix}
        }^{\sigma}
        =
        \det \roundbr[\Bigg]{
            \begin{bmatrix}
                a^{\sigma} & b^{\sigma, \divd} \\ p \cdot c^{\sigma} & d^{\sigma}
            \end{bmatrix}
        } \, .
    \]
    When the normal decompositions are free then this is true by the Leibniz determinant formula and the identity $p \cdot \sigma^{\divd}(x) = \sigma(x)$ for $x \in \calI$; the general case then follows.

    The functor in the second part of the Lemma is now given by sending $\calP$ to the tuple $(\calL, \Psi)$ with $\calL = \det(\calS \otimes_{(\calS, \calI)} \calP)$ and
    \[
        \Psi \colon \det \roundbr[\big]{\calS \otimes_{(\calS, \calI)} \calP}^{\sigma} \to \det \roundbr[\big]{\calP^{\sigma}} \to \det \roundbr[\big]{S \otimes_{(\calS, \calI)} \calP} \simeq S \otimes_{\calS} \det \roundbr[\big]{\calS \otimes_{(\calS, \calI)} \calP} \, .
    \]
    This finishes the proof.
\end{proof}

\begin{lemma} \label{lem:local-systems-witt-vectors}
    Let $R$ be a $p$-nilpotent ring.
    Then we have a natural equivalence of categories
    \[
        \set[\Bigg]{(\calL, \Psi)}{\begin{gathered} \text{$\calL$ an invertible $W(R)$-module,} \\ \text{$\Psi \colon \calL^{F} \to \calL$ an isomorphism} \end{gathered}} \to \curlybr[\Big]{\text{$\Zp$-local systems of rank $1$ on $\Spec(R)$}} \, .
    \]
    When $R$ is an $\Fp$-algebra then we similarly have a natural equivalence of categories
    \[
        \set[\Bigg]{(\calL, \Psi)}{\begin{gathered} \text{$\calL$ an invertible $W_n(R)$-module,} \\ \text{$\Psi \colon \calL^{F} \to \calL$ an isomorphism} \end{gathered}} \to \curlybr[\Big]{\text{$\ZZ/p^n$-local systems of rank $1$ on $\Spec(R)$}} \, .
    \]
\end{lemma}

\begin{proof}
    First observe that we have natural isomorphisms
    \[
        W(R)^{F = \id} \simeq W(R/p)^{F = \id} \simeq \Cont \roundbr[\big]{\abs{\Spec(R)}, \Zp}
    \]
    and, when $R$ is an $\Fp$-algebra, $W_n(R)^{F = \id} \simeq \Cont(\abs{\Spec(R)}, \ZZ/p^n)$.
    By descent it thus suffices to show that every object $(\calL, \Psi)$ as above is fpqc-locally trivial.
    This follows because the morphism
    \[
        L^+ \bbG_m \to L^+ \bbG_m, \qquad x \mapsto F(x)/x
    \]
    of affine group schemes over $\Spf(\Zp)$ is surjective.
\end{proof}

Recall that associated to a morphism of $\Zp$-local systems $f \colon \bbL \to \bbL'$ of rank $1$ on $\Spec(R)$ there is a continuous function
\[
    \nu_p(f) \colon \abs[\big]{\Spec(R)} \to \ZZ_{\geq 0} \cup \curlybr{\infty}
\]
that is induced by the $p$-adic valuation on $\Zp$.
Similarly, associated to a morphism of $\ZZ/p^n$-local systems $f \colon \bbL \to \bbL'$ of rank $1$ on $\Spec(R)$ there is a continuous function $\nu_p(f) \colon \abs{\Spec(R)} \to \curlybr{0, \dotsc, n - 1, \infty}$.

\medskip

Given a display $\calP$ over $R$ we write $\calP[1/p]$ for its associated \emph{iso-display} (see \cite[Definition 61 and Example 63]{ZinkDisplay}).
Recall that for displays $\calP$ and $\calP'$ over $R$ we have
\[
    \Hom_R \roundbr[\big]{\calP[1/p], \calP'[1/p]} \simeq \Hom_R \roundbr[\big]{\calP, \calP'}[1/p],
\]
see \cite[Proposition 66]{ZinkDisplay}.
Given displays $\calP$ and $\calP'$ over $R$ we write $f \colon \calP \dashrightarrow \calP'$ to indicate that $f$ is a morphism of iso-displays $\calP[1/p] \to \calP'[1/p]$ and say that $f$ is a \emph{morphism of displays up to isogeny}.

\begin{lemma} \label{lem:isogeny-equivalent}
    Let $\calP$ and $\calP'$ be displays over $R$ of the same constant height and dimension and let $f \colon \calP \to \calP'$ be a morphism.
    Then the following conditions are equivalent:
    \begin{enumerate}
        \item \label{item:isogeny-equivalent-1}
        $f$ induces an isomorphism of iso-displays $\calP[1/p] \to \calP'[1/p]$.

        \item \label{item:isogeny-equivalent-2}
        $\nu_p(\det(f))(x) < \infty$ for every $x \in \abs{\Spec(R)}$ (where we view $\det(f)$ as a morphism of $\Zp$-local systems).
    \end{enumerate}
\end{lemma}

\begin{proof}
    A morphism of iso-displays over $R$ is an isomorphism if and only if it is an isomorphism of underlying (finite-projective) $W(R)[1/p]$-modules.
    This property can be checked on the determinant.
    Thus we obtain
    \[
        (\ref{item:isogeny-equivalent-1})
        \Longleftrightarrow
        \roundbr[\Big]{\text{$\det(f)[1/p]$ is an isomorphism}}
        \Longleftrightarrow
        (\ref{item:isogeny-equivalent-2})
    \]
    as desired.
\end{proof}

\begin{definition}
    Let $\calP$ and $\calP'$ be displays over $R$ of the same constant height and dimension.
    Then a morphism $f \colon \calP \to \calP'$ is called an \emph{isogeny} if it satisfies the equivalent conditions from \Cref{lem:isogeny-equivalent}.
    We say that $f$ is \emph{of height $m$} if $\nu_p(\det(f)) = m$.
    Similarly, a \emph{quasi-isogeny} $f \colon \calP \dashrightarrow \calP'$ is an isomorphism of iso-displays $\calP[1/p] \to \calP'[1/p]$.

    Now let $\calP$, $\calP'$ be $n$-truncated displays over $R$ of the same constant height and dimension and let $m < n$.
    Then a morphism $f \colon \calP \to \calP'$ is called an \emph{isogeny of height $m$} if we have $\nu_p(\det(f)) = m$.
\end{definition}

\begin{remark} \label{rmk:dividing-det}
    Let $f \colon \calP \to \calP'$ be a morphism of displays over $R$ of the same constant height and dimension.
    Then it follows from the definition and \Cref{lem:local-systems-witt-vectors} that $f$ is an isogeny of height $m$ if and only if we can write
    \[
        \det(f) = p^m \cdot \varepsilon
    \]
    for some (uniquely determined) isomorphism $\varepsilon \colon \det(\calP) \to \det(\calP')$.

    Similarly, a morphism $f \colon \calP \to \calP'$ of $n$-truncated displays over an $\Fp$-algebra $R$ of the same constant height and dimension is an isogeny of height $m < n$ if and only if we can write
    \[
        \det \roundbr[\big]{f_{R/p}} = p^m \cdot \varepsilon
    \]
    for some (uniquely determined) isomorphism $\varepsilon \colon \det(\calP_{R/p})[n - m] \to \det(\calP'_{R/p})[n - m]$.
\end{remark}

\begin{remark}
    Let $\mathcal{P},\mathcal{P}^{\prime}$ be two nilpotent displays of the same height over $R$. 
	Then $f \colon \mathcal{P} \rightarrow \mathcal{P}^{\prime}$ is an isogeny of displays if and only if $\BT(f) \colon \BT(\mathcal{P})\rightarrow \BT(\mathcal{P}^{\prime})$ is an isogeny of $p$-divisible groups.
	This follows from \cite[Proposition 66]{ZinkDisplay} and by the fully faithfulness part of the equivalence of categories between nilpotent displays and formal $p$-divisible groups established in \cite{LauInventiones}.
\end{remark}

\medskip

We now record two consequences of the results from \Cref{sec:morphisms}.
To make sense of the following proposition, recall that for nilpotent displays $\calP$ and $\calP'$ over $R$ the morphisms $\Hom_R(\calP, \calP')$ are $p$-torsionfree by \Cref{prop:p-torsionfree} and thus inject into $\Hom_R(\calP[1/p], \calP'[1/p])$.

\begin{proposition} \label{cor:locus-isogeny}
    Let $f \colon \calP \dashrightarrow \calP'$ be a morphism of nilpotent displays over $R$ up to isogeny.
    Then the functor on $R$-algebras
    \[
        S \mapsto \begin{cases}
            \curlybr{\ast} & \text{if $f_S \colon \calP_S \dashrightarrow \calP'_S$ is a morphism of displays,}
            \\
            \varnothing & \text{else}
        \end{cases}
    \]
    is representable by $\Spec(R/J)$ for a finitely generated ideal $J \subseteq R$.
\end{proposition}

\begin{proof}
    Choose a non-negative integer $r$ such that $p^r f \colon \calP \dashrightarrow \calP'$ is an isogeny.
    Then for an $R$-algebra $S$ we have
    \[
        \roundbr[\Big]{\text{$f_S \colon \calP_S \dashrightarrow \calP'_S$ is a morphism of displays}} \Longleftrightarrow \roundbr[\Big]{\text{$(p^r f)_S \colon \calP_S \to \calP'_S$ is divisible by $p^r$}}
    \]
    so that we can conclude by \Cref{prop:locus-divisible-by-p}.
\end{proof}

\begin{corollary} \label{cor:inverting-isogenies}
    Let $f \colon \calP \to \calP'$ be an isogeny of nilpotent displays over $R$ of height $m$.
    Then the quasi-isogeny
    \[
        p^m \cdot f^{-1} \colon \calP' \dashrightarrow \calP
    \]
    is an isogeny.
\end{corollary}

\begin{proof}
    We already know from \Cref{cor:locus-isogeny} that the locus in $\Spec(R)$ where $p^m f^{-1}$ is an isogeny is cut out by a finitely generated ideal $I \subseteq R$.
    Now it is clear that $p^m f^{-1}$ is an isogeny when base changed to $(R/p)^{\perf}$ so that $I$ is in fact a nilpotent ideal.
    In this situation the argument from the proof of \cite[Proposition 17.6.4]{ZinkSFB} can be applied to see that $I = 0$ as desired.
\end{proof}

    \section{Proof of representability}

        \subsection{Moduli problem}
        Fix a perfect field $\F$ and a nilpotent display $\calP_0$ over $\F$ of nilpotency order $e_0$.
Then we consider the moduli problem
\[
	\calM \colon \Nilp_{W(\F)} \rightarrow \Set, \qquad
    R \mapsto \set[\Big]{(\calP, \rho)}{\text{$\calP \in \Disp(R)$, $\rho \colon \calP_{R/p} \dashrightarrow \calP_{0, R/p}$ a quasi-isogeny}} \, .
\]
Fix a lift $\widetilde{\calP}_0$ of $\calP_{0}$ to a nilpotent display over $\Spf(W(\F))$.
Then $\calM$ equivalently classifies tuples $(\calP, \rho)$ where $\rho$ is now a quasi-isogeny $\calP \dashrightarrow \widetilde{\calP}_{0, R}$.
It follows from Zink's Witt descent \cite[Theorem 37]{ZinkDisplay} that $\calM$ is a sheaf for the fpqc-topology.

We also consider the subfunctor
\[
	\calM^r \subseteq \calM
\]
parametrizing those $(\calP, \rho)$ such that $p^r \rho \colon \calP \to \widetilde{\calP}_{0, R}$ is an isogeny of displays.
By \Cref{cor:locus-isogeny} the inclusion $\calM^r \to \calM$ is a representable finitely presented closed immersion.

Note that $\calM^{r}$ further decomposes as a disjoint union
\[
    \calM^r = \bigsqcup_{s \geq 0} \calM^{r, s}
\]
where $\calM^{r, s} \subseteq \calM^r$ is given by the condition that $p^r \rho$ is an isogeny of height $s$.

        \subsection{Representability of $\calM^r$}
        In this section we show the representability of $\calM^r$ by a $p$-adic formal scheme.

\begin{theorem}[Artin's criterion] \label{thm:artin}
    Let $X \colon \Nilp_{W(\F)} \to \Set$ be an étale sheaf that satisfies the following conditions:
    \begin{enumerate}
        \item \label{item:artin-diagonal}
        The diagonal $\Delta \colon X \to X \times X$ is a closed immersion.

        \item \label{item:artin-rs}
        Let
        \[
        \begin{tikzcd}
            & R_2 \ar[d]
            \\
            R_1 \ar[r]
            & R
        \end{tikzcd}
        \]
        be a diagram of $p$-nilpotent $W(\F)$-algebras such that $R_1 \to R$ is surjective with nilpotent kernel.
        Then the natural map
        \[
            X \roundbr[\big]{R_1 \times_R R_2} \to X(R_1) \times_{X(R)} X(R_2)
        \]
        is bijective.

        \item \label{item:artin-tangent}
        Given a finite extension $\F'$ of $\F$ and a point $x \in X(\F')$ the tangent space $T_x X$ has finite dimension.

        \item \label{item:artin-algebraize}
        Let $A$ be a $p$-nilpotent first countable admissible linearly topologized $W(\F)$-algebra.
        Then the natural map
        \[
            X(A) \to X \roundbr[\big]{\Spf(A)}
        \]
        is bijective.

        \item \label{item:artin-lofp}
        $X$ is locally of finite presentation (i.e.\ it commutes with filtered colimits, see \cite[Tag 049J]{stacks}).
    \end{enumerate}
    Then $X$ is a $p$-adic formal algebraic space that is separated and locally of finite type over $\Spf(W(\F))$.
\end{theorem}

\begin{proof}
    This follows formally from \cite[Tags 07Y1 and 0CXU]{stacks}.
\end{proof}

\begin{proposition}[$\calM^r$ is representable] \label{prop:m-r-representable}
    The functor $\calM^r$ is representable by a $p$-adic formal algebraic space that is separated and locally of finite type over $\Spf(W(\F))$.
\end{proposition}

\begin{proof}
    We check that the conditions from \Cref{thm:artin} are satisfied.

    $(\ref{item:artin-diagonal})$:
    Let $(\calP, \rho), (\calP', \rho') \in \calM(R)$.
    Then the locus in $\Spec(R)$ where these two points agree is precisely the locus where $\rho'^{-1} \circ \rho$ and $\rho^{-1} \circ \rho'$ both are isogenies.
    Thus by \Cref{cor:locus-isogeny} it is cut out by a finitely generated ideal.
    This shows that the diagonal $\calM \to \calM \times \calM$ is a (finitely presented) closed immersion.

    $(\ref{item:artin-rs})$:
    Let $R$, $R_1$, $R_2$ be as in the theorem.
    By Ferrand gluing \cite[Théorème 2.2]{Ferrand} we have equivalences of categories
	\[
		\Disp \roundbr[\big]{R_{1} \times_{R} R_{2}} \to \Disp(R_{1}) \times_{\Disp(R)} \Disp(R_{2})
    \]
    and
    \[
        \Iso \roundbr[\big]{R_{1} \times_{R} R_{2}} \to \Iso(R_{1}) \times_{\Iso(R)} \Iso(R_{2}) \, .
	\]
    From this it then formally follows that $\calM^r(R_1 \times_R R_2) \to \calM^r(R_1) \times_{\calM^r(R)} \calM^r(R_2)$ is a bijection, in fact that the same statement is true for $\calM$.

    $(\ref{item:artin-tangent})$:
    Let $x = (\calP, \rho) \in \calM(R)$ and write $(M, \Fil(M))$ its underlying filtered finite projective $R$-module.
    Then by \cite[Equation 83]{ZinkDisplay} and the unique liftability of quasi-isogenies we have a natural isomorphism
    \[
        \Hom_R(\Fil(M), M/\Fil(M)) \to T_x \calM
    \]
    so that $T_x \calM$ is a finite projective $R$-module.

    $(\ref{item:artin-algebraize})$:
    Suppose we are given an extension problem as in the theorem.
    By \cite[Lemma 2.10]{LauFrames} we have an equivalence of categories
    \[
        \Disp(A) \to \lim_{\text{$I \subseteq A$ open}} \Disp(A/I) \, .
    \]
    To deduce that $\calM^r(A) \to \calM^r(\Spf(A))$ is bijective we now still need to check that a morphism $f \colon \calP \to \calP'$ of displays over $A$ is an isogeny if and only if the induced morphism $f_{A/I} \colon \calP_{A/I} \to \calP'_{A/I}$ is an isogeny for all open ideals $I \subseteq A$.
    This follows because the locus where $f$ is an isogeny is representable by an open subscheme of $\Spec(A)$.

    \smallskip

    It now only remains to see that condition $(\ref{item:artin-lofp})$ is satisfied.
    This occupies the rest of this section and is finally proven in \Cref{prop:lofp} below.
\end{proof}

\begin{lemma} \label{lem:lofp-on-objects}
    Let $X \colon \Nilp_{W(\F)} \to \Set$ be an fpqc-sheaf that satisfies the following conditions:
    \begin{enumerate}
        \item \label{item:lofp-on-objects-diagonal}
        The diagonal $\Delta \colon X \to X \times X$ is locally of finite presentation.

        \item \label{item:lofp-on-objects}
        Let $R = \colim_{i \in I} R_i$ be a filtered colimit of $p$-nilpotent $W(\F)$-algebras.
        Then for every point $x \in X(R)$ there exists an index $i \in I$ and a faithfully flat ring homomorphism $R_i \to S_i$ such that
        \[
            x_{S_i \otimes_{R_i} R} \in \im \roundbr[\Big]{X(S_i) \to X \roundbr[\big]{S_i \otimes_{R_i} R}} \, .
        \]
    \end{enumerate}
    Then $X$ is locally of finite presentation.
\end{lemma}

\begin{proof}
    The proof given in \cite[Lemma 2.1.6]{EmertonGeeImages}, which concerns stacks for the étale topology, adapts to our situation.
\end{proof}

\begin{lemma} \label{lem: covering rings by semiperfect rings}
	Let $R$ be a ring.
	There exists a faithfully flat ring homomorphism $R\rightarrow S$, such that $S/p$ is semiperfect.
\end{lemma}

\begin{proof}
	Choose a surjection $P = \mathbb{Z} \polring{X_i}{i \in I} \rightarrow R$ for some set $I$.
	Now observe that
    \[
        P \rightarrow P^{1/p^{\infty}} \coloneqq \mathbb{Z} \polring[\big]{X_i^{1/p^{\infty}}}{i \in I}
    \]
    is faithfully flat and that $P^{1/p^{\infty}}/p$ is perfect.
	Thus $R \rightarrow S \coloneqq R \otimes_P P^{1/p^{\infty}}$ is faithfully flat and $S/p$ is semiperfect, as desired.
\end{proof}

Note that for a semiperfect $\Fp$-algebra $R$ the Witt vector Frobenius is surjective on $W(R)$ and consequently $I_R = p W(R)$.
From this it follows that $W(R)/p^n \simeq W_n(R)$.
This will be implicitly used in the proof of the following lemma.

\begin{lemma} \label{lem:extending-adjoint-matrix}
    Let $R$ be a $W(\F)/p^a$-algebra and fix non-negative integers $m$ and $e \geq e_0$.
    There exists a constant $\const = \const(a, m, e)$, depending only on $a$, $m$ and $e$, with the following property:

    Let $n > m$ and let $\calP_{n + \const}$ be an $(n + \const)$-truncated display, nilpotent of nilpotency order less or equal than $e$, and let
    \[
        g_{n + \const} \colon \calP_{n + \const} \to \widetilde{\calP}_{0, R}[n + \const]
    \]
    be an isogeny of $(n + \const)$-truncated displays of height $m$.
    Then there exist the following data.
    \begin{itemize}
        \item
        A faithfully flat ring homomorphism $R \to S$.

        \item
        A display $\calP$ over $S$ and an isogeny of displays $g \colon \calP \to \widetilde{\calP}_{0, S}$ of height $m$.

        \item
        An isomorphism $f_n \colon \calP_{n + \const, S}[n] \to \calP[n]$ of $n$-truncated displays over $S$ that makes the diagram
        \[
        \begin{tikzcd}[column sep = large]
            \calP_{n + \const, S}[n] \ar[r, "{g_{n + \const}[n]}"] \ar[d, "{f_n}"]
            & \widetilde{\calP}_{0, S}[n] \ar[d, equal]
            \\
            \calP[n] \ar[r, "{g[n]}"]
            & \widetilde{\calP}_{0, S}[n]
        \end{tikzcd}
        \]
        commutative.
    \end{itemize}
\end{lemma}

\begin{proof}
    Note that the condition that $g$ is an isogeny of height $m$ is superfluous; it automatically follows from the third condition and the assumption $n > m$.
    Moreover we can and do always assume without loss of generality that $\calP_{n + \const}$ has a free normal decomposition.

    Let us first assume that $R$ is a semiperfect $\F$-algebra, set $\const \coloneqq m$ and suppose that we are given $\calP_{n + m}$ and $g_{n + m}$ as above.
    We can represent $\calP_{n + \const}$, $\widetilde{\calP}_0$ and $g_{n + \const}$ by matrices
    \[
        U_{n + m} \in \GL_h \roundbr[\big]{W_{n + m}(R)}, \quad
        \widetilde{U}^0 \in \GL_h \roundbr[\big]{W(W(\F))}
        \quad \text{and} \quad
        M_{n + m} \in \Mat_h \roundbr[\big]{W_{n + m}(R)}_{\mu}
    \]
    satisfying the structure equation $M_{n + m} \cdot U_{n + m} = \widetilde{U}^0 \cdot \Phi(M_{n + m}) \in \Mat_h(W_{n + m}(R))$.
    Our goal is now to find matrices $U \in \GL_h \roundbr[\big]{W(R)}$ and $M \in \Mat_h \roundbr[\big]{W(R)}_{\mu}$ such that $U[n] = U_{n + m}[n]$, $M[n] = M_{n + m}[n]$ and
    \[ \label{eq:structure-equation}
        M \cdot U = \widetilde{U}^0 \cdot \Phi(M) \in \Mat_h \roundbr[\big]{W(R)} \, . \tag{$\dagger$}
    \]
    For this we start by choosing arbitrary lifts
    \[
        \text{$U^{\approxi} \in \GL_h \roundbr[\big]{W(R)}$ of $U_{n + m}$}
        \qquad \text{and} \qquad
        \text{$M \in \Mat_h \roundbr[\big]{W(R)}_{\mu}$ of $M_{n + m}$} \, .
    \]
    From our assumption that $g_{n + m}$ is an isogeny of height $m$ it follows that we can write
    \[
        \det(M) = p^m \varepsilon
    \]
    for some unit $\varepsilon \in W(R)^{\times}$, see \Cref{rmk:dividing-det}.
    The idea is now to find a suitable modification $U$ of $U^{\approxi}$ that solves the equation $(\text{\ref{eq:structure-equation}})$.
    We find a matrix $X \in \Mat_{h}(W(R))$ such that
	\[
		M \cdot U^{\approxi} + p^{n + m} \cdot X = \widetilde{U}^0 \cdot \Phi(M) \, .
	\]
	Let $M^{\ad} \in \Mat_{h}(W(R))$ be the adjoint matrix of $M$; it has the property that $M \cdot M^{\ad} = p^m \varepsilon \cdot \mathmybb{1}_h$.
	Now set
	\[
		U \coloneqq U^{\approxi} + p^n \varepsilon^{-1} \cdot M^{\ad} \cdot X \in \GL_{h} \roundbr[\big]{W(R)} \, .
	\]
    Then we have
    \begin{align*}
		M \cdot U &= M \cdot U^{\approxi} + p^n \varepsilon^{-1} \cdot M \cdot M^{\ad} \cdot X \\
	  	&= M \cdot U^{\approxi} + p^{n + m} \cdot X \\
	  	&= \widetilde{U}^0 \cdot \Phi(M)
	\end{align*}
    as desired.

    Now let $R$ be arbitrary, set $\const \coloneqq m + \const_1$ where $\const_1 = \const_1(a - 1, e)$ is the constant from \Cref{cor: Corollary Def statement via truncated displays} and suppose again that we are given $\calP_{n + \const}$ and $g_{n + \const}$ as above.
    Choose a faithfully flat ring homomorphism $R \to S$ such that $S/p$ is semiperfect, see \Cref{lem: covering rings by semiperfect rings}, and equip $pS \subseteq S$ with its canonical pd-structure so that $S \to S/p$ becomes a pd-thickening.

    Combining the first part of the proof with \Cref{cor: Corollary Def statement via truncated displays} we then obtain data $\calP$, $g$ and $f_n$ satisfying the conditions above.
\end{proof}

\begin{lemma} \label{cor:lifting-isomorphism}
    Let $R$ be a $W(\F)/p^a$-algebra and consider a lifting problem
    \[ \label{eq:lifting-isomorphism}
    \begin{tikzcd}
        \calP \ar[r, "g"] \ar[d, dashed]
        & \widetilde{\calP}_{0, R} \ar[d, equal]
        \\
        \calP' \ar[r, "g'"]
        & \widetilde{\calP}_{0, R}
    \end{tikzcd} \tag{$\ast$}
    \]
    where $\calP$ and $\calP'$ are displays over $R$ that are nilpotent of nilpotency order less or equal than some $e \geq e_0$ and $g$ and $g'$ are isogenies of the same height $m$.
    Then there exists a non-negative integer $n = n(a, e, c, m)$, depending only on $a$, $e$, $m$ and the codimension $c$ of $\calP_0$, such that the following conditions are equivalent:
    \begin{enumerate}
        \item
        The lifting problem $(\text{\ref{eq:lifting-isomorphism}})$ has a solution $f \colon \calP \to \calP'$.

        \item
        The truncated lifting problem $(\text{\ref{eq:lifting-isomorphism}})[n]$ has a solution $f_n \colon \calP[n] \to \calP'[n]$.
    \end{enumerate}
    When these conditions are satisfied then the solution $f$ is unique and an isomorphism.
\end{lemma}

\begin{proof}
    \enquote{$(1) \Longrightarrow (2)$}:
    This is clear.

    \enquote{$(2) \Longrightarrow (1)$}:
    By \Cref{cor:inverting-isogenies} we have the isogeny
    \[
        \adjust{f} \coloneqq \roundbr[\big]{p^m \cdot g''^{-1}} \circ g' \colon \calP' \to \calP''.
    \]
    Now condition $(1)$ is equivalent to requiring $\adjust{f}$ to be divisible by $p^m$ and similarly condition $(2)$ implies that $\adjust{f}[n]$ is divisible by $p^m$.
    Thus we can conclude by \Cref{prop:divisible-by-p-truncated}.
\end{proof}

\begin{proposition} \label{prop:lofp}
    $\calM^r$ is locally of finite presentation.
\end{proposition}

\begin{proof}
    We apply \Cref{lem:lofp-on-objects}.
    As we have already shown that the diagonal $\Delta \colon \calM^r \to \calM^r \times \calM^r$ is a finitely presented closed immersion our task is to verify condition $(\ref{item:lofp-on-objects})$ of the lemma.

    So let $R = \colim_{i\in I} R_{i}$ be a filtered colimit of $p$-nilpotent $W(\F)$-algebras, let $(\calP, \rho) \in \calM^r(R)$ and write $g \coloneqq p^{r} \rho \colon \calP \to \widetilde{\calP}_{0, R}$.
    After passing to a Zariski-cover of $\Spec(R)$ we may assume that $g$ is an isogeny of constant height $m$.
	Let $e$ be the maximum of the nilpotency orders of $\widetilde{\calP}_{0}$ and $\calP$ and let $\const = \const(a, e, m)$ and $n = n(a, e, m, c)$ be the constants from \Cref{lem:extending-adjoint-matrix} and \Cref{cor:lifting-isomorphism}, where $c$ denotes the codimension of $\calP_0$.

	The isogeny $g \colon \calP \rightarrow \widetilde{\calP}_{0, R}$ induces an isogeny of $(n + \const)$-truncated displays
	\[
		g[n + \const] \colon \calP[n + \const] \rightarrow \widetilde{\calP}_{0,R}[n + \const] \, .
	\]
	We may use \Cref{lem: abgeschnittene Displays ueber Kolimes sind Kolimes} to descend $\calP[n + \const]$ and $g[n + \const]$ to
	\[
		g_{i, n + \const} \colon \calP_{i, n + \const} \rightarrow \widetilde{\calP}_{0, R_i}[n + \const]
	\]
	for some $i \in I$.
	Upon increasing the index $i \in I$ if necessary, we may assume that $g_{i, n + \const}$ is again an isogeny of height $m$ and that the nilpotency order of $\calP_{i, n + \const}$ is still less or equal than $e$.
	We apply \Cref{lem:extending-adjoint-matrix} to find a faithfully flat ring map $R_{i} \rightarrow S_{i}$ and data
    \[
        \calP_i, \quad
        g_i \colon \calP_i \to \widetilde{\calP}_{0, S_i} \quad \text{and} \quad
        f_{i, n} \colon \calP_{i, n + \const, S_i}[n] \to \calP[n]
    \]
    as in the statement.
	Setting $\rho_i = p^{-r} g_i$ we then obtain a point $(\calP_{i}, \rho_{i}) \in \calM^r(S_{i})$.

	From \Cref{cor:lifting-isomorphism} it then follows that the points $(\calP^{\prime}_{i},\rho^{\prime}_{i}) \in \calM^r(S_{i})$ and $(\calP, \rho) \in \calM^r(R)$ are identified in $\calM^r(S_i \otimes_{R_i} R)$ as desired.
\end{proof}

Using the projectivity of the Witt vector affine Grassmannian \cite{BhattScholzeWitt} we can even show the representability of $\calM^r$ as a $p$-adic formal scheme.

\begin{corollary} \label{cor: Corollary Mr,s are p-adic formal schemes}
    $\calM^{r, s}$ is a $p$-adic formal scheme that is proper over $\Spf(W(\F))$.
\end{corollary}

\begin{proof}
    Using \cite[Tag 07VV]{stacks} we are reduced to showing that the perfection $\calM^{r, s, \perf}$ is a perfect scheme perfectly proper over $\F$.
    This follows from \cite[Theorem 1.1]{BhattScholzeWitt}.
\end{proof}

        \subsection{Representability of $\calM^{\Zink}$}
        Having shown the representability of $\calM^r$ we now turn to the representability of the whole moduli space $\calM$.

\begin{proposition} \label{prop:abstract-ind-formal}
    Let $X$ be an ind-scheme and suppose that the following conditions are satisfied:
    \begin{enumerate}
        \item \label{item:abstract-ind-formal-1}
        $X$ can be written as a countably indexed filtered colimit of locally Noetherian schemes along closed immersions.

        \item \label{item:abstract-ind-formal-2}
        The underlying topological space $\abs{X}$ is a locally Noetherian locally spectral space.
        
        \item \label{item:abstract-ind-formal-3}
        Let $A$ be an admissible linearly topologized ring with nilpotent ideal of definition.
        Then the natural map
        \[
            X(A) \to X \roundbr[\big]{\Spf(A)}
        \]
        is bijective.
    \end{enumerate}
    Then $X$ is a locally Noetherian adic formal scheme.
\end{proposition}

\begin{proof}
    Write $X = \colim_{i \geq 1} X_i$ as in condition $(\text{\ref{item:abstract-ind-formal-1}})$ and let $x \in \abs{X}$.
    Our goal is to find an open immersion $\Spf(B) \to X$ whose topological image contains $x$ and such that $B$ is a Noetherian adic ring.

    Using condition $(\text{\ref{item:abstract-ind-formal-2}})$ we may assume without loss of generality that $\abs{X_1}$ already contains the (finitely many) irreducible components of $\abs{X}$ containing $x$.
    Now choose an affine neighorhood $U_1 = \Spec(B_1) \subseteq X_1$ of $x$.
    After possibly shrinking $U_1$ we may then assume that $\abs{U_1}$ does not intersect any irreducible component of $\abs{X}$ that does not contain $x$.
    This implies that $\abs{U_1} \subseteq \abs{X}$ is open.
    We thus obtain open subschemes $U_i = \Spec(B_i) \subseteq X_i$ with $\abs{U_i} = \abs{U_1}$ for all $i \geq 1$.
    These fit into a commutative diagram
    \[
    \begin{tikzcd}
        U_1 \ar[r] \ar[d]
        & U_2 \ar[r] \ar[d]
        & U_3 \ar[r] \ar[d]
        & \hdots
        \\
        X_1 \ar[r]
        & X_2 \ar[r]
        & X_3 \ar[r]
        & \hdots
    \end{tikzcd}
    \]
    where the upper horizontal morphisms are thickenings and the squares are pullbacks.
    The limit $B = \lim_i B_i$ now is an admissible linearly topologized ring and we have an open immersion $U = \Spf(B) \to X$ with topological image $\abs{U_1}$.
    
    Now choose an ideal of definition $I \subseteq B$.
    Note that condition $(\text{\ref{item:abstract-ind-formal-3}})$ implies that the same condition is satisfied if we replace $X$ by $\Spf(B)$.
    We can thus apply it to $A = B/\overline{I^n}$ to see that the closure $\overline{I^n} \subseteq B$ is an open ideal.
    Thus $B$ is weakly adic and from \cite[Tag 0AKM]{stacks} it follows that $B$ is a Noetherian adic ring as desired.
\end{proof}

\begin{corollary}\label{cor: Corollary how to check M is representable}
    Let $X$ be an ind-scheme over $\Spf(W(\F))$ and assume that the following conditions are satisfied:
    \begin{itemize}
        \item
        $X$ can be written as a countably indexed filtered colimit of $p$-adic formal schemes that are locally of finite type over $\Spf(W(\F))$ along closed immersions.

        \item
        $\abs{X}$ is a locally Noetherian locally spectral space.

        \item
        Let $A$ be an admissible linearly topologized $W(\F)$-algebra with nilpotent ideal of definition.
        Then the natural map
        \[
            X(A) \to X \roundbr[\big]{\Spf(A)}
        \]
        is bijective.
    \end{itemize}
    Then $X$ is a formal scheme locally formally of finite type over $\Spf(W(\F))$.
\end{corollary}

\begin{proof}
    This follows from \Cref{prop:abstract-ind-formal}.
\end{proof}

\begin{theorem}[$\calM$ is representable] \label{thm:main-result}
    The functor $\calM$ is representable by a formal scheme that is locally formally of finite type, ind-proper and formally smooth over $\Spf(W(\F))$.
\end{theorem}

\begin{proof}
	We want to apply \Cref{cor: Corollary how to check M is representable}; to do so we have to verify that the assumptions are met for $\calM$. 
	By \Cref{prop:m-r-representable} and \Cref{cor:locus-isogeny}, we may write
	\[
		\calM=\colim_{r} \calM^{r},
	\]
	where the transition maps are closed immersions of $p$-adic formal schemes. 
	It follows from \cite[Thm. 3.1]{ZhuWitt} that $\abs{\calM}$ is a locally Noetherian locally spectral space. 
	To verify that every extension problem is solvable we have to algebraize a point $(\calP,\rho) \in \calM(\Spf(A))$, for $A$ an admissible linearly topologized $W(\F)$-algebra with nilpotent ideal of definiton, to a point in $\calM(A)$.
	One may use \cite[Lemma 2.10]{LauFrames} to algebraize the display and the unique liftability of quasi-isogenies to algebraize the quasi-isogeny.

    The formal smoothness of $\calM$ then follows from the deformation theory of nilpotent displays.
\end{proof}

        \subsection{Ampleness of the Hodge-bundle}
        We include the following proposition, showing that the Hodge bundle on $\calM$ is ample (after restriction to a quasi-compact closed subscheme).

S.B.\ thanks Laurent Fargues for crucial discussions about this result; it was originally found in the beginning of his PhD.

\begin{proposition}[ampleness of the Hodge bundle] \label{prop: Local ampelness}
    Let $(P^{\univ}, Q^{\univ})$ be the underlying $(W(\calO_{\calM}), I_{\calO_{\calM}})$-module of the universal nilpotent display over $\calM$ (where we use the notation from \cite{ZinkDisplay}).
    Then the restriction of the Hodge bundle
    \[
        \omega \coloneqq \det \roundbr[\big]{P^{\univ}/Q^{\univ}}^{\vee} \in \Pic(\calM)
    \]
    to any closed $p$-adic formal subscheme $X \subseteq \calM$ of finite type over $\Spf(W(\F))$ is (relatively) ample.
    Consequently every such $X$ admits a closed immersion $X \to \bbP^N_{\Spf(W(\F))}$ for some $N$.
\end{proposition}

\begin{proof}
    We may assume without loss of generality that $X$ is an $\F$-scheme, see \cite[Théorème 4.7.1]{EGA3.1}.
    After translating along $(X, \rho) \mapsto (X, p^r \rho)$ for $r \gg 0$ we may also assume that $X \subseteq \calM^0$.
    By \cite[Lemma 3.6]{BhattScholzeWitt} it now suffices to show that the restriction $\restr{\omega}{X^{\perf}}$ is ample.

    Now recall that Bhatt-Scholze define, for a perfect $\Fp$-algebra $R$, a natural functor
    \[
        \widetilde{\det} \colon \Perf \roundbr[\big]{\text{$W(R)$ on $R$}}^{\simeq} \longrightarrow \Pic^{\ZZ}(R)
    \]
    from the core of the $\infty$-category $\Perf(\text{$W(R)$ on $R$})$ of perfect complexes of $W(R)$-modules which are acyclic after inverting $p$ to the groupoid $\Pic^{\ZZ}(R)$ of graded invertible $R$-modules, see \cite[Theorem 5.7]{BhattScholzeWitt}.
    It extends the determinant functor $\det \colon \Perf(R)^{\simeq} \to \Pic^{\ZZ}(R)$ and is multiplicative in fiber sequences.

    Let $P_0$ be the underlying finite projective $W(\F)$-module of the display $\calP_0$ and set
    \[
        \Gr^{\Witt, 0} \colon \Alg_{\F}^{\perf} \to \Set, \qquad
        R \mapsto \curlybr[\big]{\text{isogenies $\rho \colon P \to P_{0, W(R)}$ of finite projective $W(R)$-modules}} \, .
    \]
    This is (a part of) the Witt vector affine Grassmannian attached to $P_0$; it is an ind-perfectly proper ind-perfect scheme over $\F$, see \cite[Theorem 8.3]{BhattScholzeWitt}.
    Moreover, the line bundle
    \[
        \calL \coloneqq \widetilde{\det} \roundbr[\big]{P_{0, \calO_{\Gr}}/P^{\univ}} \in \Pic \roundbr[\big]{\Gr^{\Witt, 0}}
    \]
    is ample after restricting to any quasi-compact closed perfect subscheme of $\Gr^{\Witt, 0}$.

    We now claim that there is a natural isomorphism
    \[
        \restr{\omega^{\otimes p}}{\calM^{0, \perf}} \simeq i^* \calL^{\otimes (p - 1)},
    \]
    where $i \colon \calM^{0, \perf} \to \Gr^{\Witt, 0}$ denotes the evident closed immersion.
    For this, let $(\calP, \rho) \in \calM^0(R)$ for some perfect $\F$-algebra $R$ and write $(P, Q)$ for the underlying $(W(R), I_R)$-module of $\calP$.
    We then have a commutative diagram
    \[
    \begin{tikzcd}
        0 \ar[r]
        & P \ar[r, "\rho"] \ar[d, "V^{\sharp}"]
        & P_{0, W(R)} \ar[r] \ar[d, "V^{\sharp}"]
        & P_{0, W(R)}/P \ar[r] \ar[d]
        & 0
        \\
        0 \ar[r]
        & P^F \ar[r, "\rho^F"]
        & P_{0, W(R)}^F \ar[r]
        & (P_{0, W(R)}/P)^F \ar[r]
        & 0
    \end{tikzcd}
    \]
    with exact rows.
    After fixing a trivialization of $\det(\coker(V^{\sharp} \colon P_0 \to P_0))$ we thus obtain a natural isomorphism
    \[
        \det \roundbr[\big]{\coker(V^{\sharp} \colon P \to P^F)} \simeq \widetilde{\det} \roundbr[\big]{P_{0, W(R)}/P} \otimes \widetilde{\det} \roundbr[\big]{(P_{0, W(R)}/P)^F}^{\vee} \, .
    \]
    The right hand side identifies with the evaluation of $i^* \calL^{\otimes (1 - p)}$ at the $R$-valued point $(\calP, \rho)$.
    On the other hand the image of $V^{\sharp} \colon P \to P^F$ is precisely the Frobenius twist $Q^F$, so that $\coker(V^{\sharp}) \simeq (P/Q)^F$.
    Thus the left hand side identifies with the evaluation of $\omega^{\otimes (-p)}$.
    This proves our claim and consequently the ampleness of $\restr{\omega}{X^{\perf}}$.

    The second part of the proposition then follows from Grothendieck's algebraization theorem, see \cite[Tag 089A]{stacks}.
\end{proof}

    \bibliography{bibliography}
    \bibliographystyle{alpha}


\end{document}